\documentclass[final,leqno]{siamltex}
\usepackage{mathbbol}
\usepackage{mathrsfs}
\usepackage{amsfonts}
\usepackage{stmaryrd}
\usepackage{amsmath}
\usepackage{amssymb}
\usepackage{multirow}
\usepackage{caption}
\usepackage{tabularx}
\usepackage{stmaryrd}
\usepackage{booktabs}
\usepackage{cite}
\usepackage{amssymb}
\usepackage{verbatim}
\usepackage{color}
\usepackage{graphicx}
\usepackage{subfigure}
\usepackage{hyperref}
\usepackage[all,cmtip]{xy}
\usepackage{epsfig,epstopdf,color,bm}
\usepackage{tabularx}

\input psfig.sty

\newcommand{\p}{\partial}
\newcommand{\og}{\omega}
\newcommand{\tht}{\theta}
\newcommand{\Og}{\Omega}
\newcommand{\ep}{\varepsilon}
\newcommand{\fl}[2]{\frac{#1}{#2}}
\newcommand{\dt}{\delta}
\newcommand{\nn}{\nonumber}
\newcommand{\Dt}{\Delta}
\renewcommand{\theequation}{\arabic{section}.\arabic{equation}}
\newcommand{\be}{\begin{equation}}
\newcommand{\ee}{\end{equation}}
\newcommand{\ba}{\begin{array}}
\newcommand{\ea}{\end{array}}
\newcommand{\bea}{\begin{eqnarray}}
\newcommand{\eea}{\end{eqnarray}}
\newcommand{\beas}{\begin{eqnarray*}}
\newcommand{\eeas}{\end{eqnarray*}}
\newtheorem{remark}{Remark}[section]
\newcommand{\bx}{{\bf x} }

\def\C{{\mathbb C}}
\def\R{{\mathbb R}}

\def\({\left(}
\def\){\right)}
\def\<{\left\langle}
\def\>{\right\rangle}

\title{Error estimates of a regularized finite difference method for the logarithmic Schr\"odinger equation\thanks{This work was partially supported by the Ministry
of Education of Singapore grant
R-146-000-223-112 (MOE2015-T2-2-146) (W. Bao).}}
\author{Weizhu Bao\thanks{Department of Mathematics,
National University of Singapore, Singapore 119076 ({\tt matbaowz@nus.edu.sg},
URL: http://www.math.nus.edu.sg/\~{}bao/)} \and
R\'{e}mi Carles\thanks{CNRS, Institut Montpelli\'erain Alexander
  Grothendieck, Univ. Montpellier, France ({\tt
    Remi.Carles@math.cnrs.fr}, URL: http://carles.perso.math.cnrs.fr)}
    \and
Chunmei Su\thanks{Department of Mathematics, University of Innsbruck, Innsbruck 6020, Austria ({\tt sucm13@163.com})}
\and
Qinglin Tang\thanks{School of Mathematics,
Sichuan University, Chengdu 610064 ({\tt qinglin\_tang@scu.edu.cn})}}

\date{}
\begin{document}

\maketitle

\begin{abstract}
We present a regularized finite difference method for the logarithmic
Schr\"odinger equation (LogSE) and establish its error bound. Due to the blow-up of the logarithmic nonlinearity, i.e. $\ln \rho\to -\infty$ when $\rho\rightarrow 0^+$ with $\rho=|u|^2$ being the density and $u$ being the complex-valued wave function or order parameter, there are significant difficulties in designing numerical methods and establishing
their error bounds for the LogSE. In order to suppress the round-off error and to avoid blow-up, a regularized logarithmic Schr\"odinger equation (RLogSE) is proposed with a small regularization parameter $0<\ep\ll 1$ and linear convergence is established between the solutions of RLogSE and LogSE in term of  $\ep$. Then a semi-implicit finite difference method is presented for discretizing the RLogSE and error estimates are established in terms of the mesh size $h$ and time step $\tau$ as well as the small regularization parameter $\ep$. Finally numerical results are reported to confirm our error bounds.
\end{abstract}


\pagestyle{myheadings}\thispagestyle{plain}
\section{Introduction}
We consider the logarithmic Schr\"odinger equation (LogSE) which arises in a model of nonlinear wave mechanics (cf. \cite{BiMy76}),
\be\label{LSE}
\left\{
\begin{aligned}
&i\p_t u(\bx,t)+\Delta u(\bx,t)=\lambda u(\bx,t)\,\ln(|u(\bx,t)|^2),\quad \bx\in \Omega, \quad t>0,\\
& u(\bx,0)=u_0(\bx),\quad \bx\in \overline{\Omega},
\end{aligned}
\right.
\ee
where $t$ is time, $\bx\in \mathbb{R}^d$ ($d=1,2,3$) is the spatial coordinate, $\lambda\in \mathbb{R}\backslash\{0\}$ measures the force of the nonlinear interaction,
$u:=u(\bx,t)\in\mathbb{C}$ is the dimensionless wave function or order parameter and
$\Omega=\mathbb{R}^d$ or $\Omega\subset\mathbb{R}^d$ is a bounded
domain with homogeneous Dirichlet or periodic boundary condition\footnote{Whenever
  we consider this case, it is assumed that the boundary is Lipschitz
  continuous.}
fixed on the boundary. It admits applications to
quantum mechanics \cite{BiMy76, BiMy79}, quantum optics \cite{buljan, KEB00}, nuclear physics \cite{Hef85}, transport and diffusion phenomena \cite{DFGL03, hansson},
open quantum systems \cite{yasue, HeRe80}, effective quantum gravity \cite{Zlo10}, theory of superfluidity and
Bose-Einstein condensation \cite{BEC}. The logarithmic Schr\"odinger equation enjoys three
conservation laws, {\sl mass}, {\sl momentum} and {\sl
  energy} \cite{CazCourant,CaHa80}, like in the case of the nonlinear Schr\"odinger
equation with a power-like nonlinearity (e.g. cubic):
\be\label{conserv}
\begin{split}
M(t):&=\|u(\cdot,t)\|_{L^2(\Omega)}^2=\int_\Omega|u(\bx,t)|^2d\bx\equiv
  \int_\Omega|u_0(\bx)|^2d\bx=M(0),\\
P(t):&=\mathrm{Im} \int_\Omega \overline{u}(\bx,t)\nabla
  u(\bx,t)d\bx\equiv \mathrm{Im} \int_\Omega \overline {u_0}(\bx)\nabla
  u_0(\bx)d\bx=P(0),\quad t\ge0,\\
E(t):&=\int_\Omega\left[|\nabla u(\bx,t)|^2d\bx+\lambda F(|u(\bx,t)|^2)\right]d\bx\\
&\equiv\int_\Omega\left[|\nabla u_0(\bx)|^2+\lambda F(|u_0(\bx)|^2)\right]d\bx=E(0),
\end{split}
\ee
where ${\rm Im}\, f$ and $\overline{f}$ denote the imaginary part and complex conjugate of $f$, respectively,  and
\be
F(\rho)=\int_0^\rho \ln(s)ds=\rho\, \ln \rho-\rho, \qquad \rho\ge0.
\ee

On a mathematical level, the logarithmic nonlinearity possesses
several features that make it quite different from more standard
nonlinear Schr\"odinger equations. First, the nonlinearity is not
locally Lipschitz continuous because of the behavior of the logarithm
function at the origin. Note that in view of numerical simulation,
this singularity of the ``nonlinear potential'' $\lambda\ln
(|u(\bx,t)|^2)$ makes the choice of a discretization quite
delicate. The second aspect is that whichever the sign of $\lambda$,
the nonlinear potential energy in $E$ has no definite sign. In fact,
whether the nonlinearity is repulsive/attractive (or defocusing/focusing)
depends on both $\lambda$ and the value of the density $\rho:=\rho(\bx,t)
=|u(\bx,t)|^2$. When $\lambda>0$, then the nonlinearity $\lambda \rho\, \ln\rho$ is
repulsive when $\rho>1$; and respectively, it is attractive when $0<\rho<1$.
On the other hand, when $\lambda<0$, then the nonlinearity $\lambda \rho\, \ln\rho$ is
attractive when $\rho>1$; and respectively, it is repulsive when $0<\rho<1$.
Therefore,
solving the Cauchy problem for \eqref{LSE} is not a trivial issue, and
constructing solutions which are defined for all time requires some
work; see \cite{CaHa80,GLN10,CaGa-p}. Essentially, the outcome is that
if $u_0$ belongs to (a subset of) $H^1(\Omega)$, \eqref{LSE} has a
unique, global solution, regardless of the space dimension $d$ (see also
Theorem~\ref{theo:cauchy} below).
\smallbreak

Next, the large time behavior reveals new phenomena. A
first remark suggests that nonlinear effects are weak. Indeed,  unlike what happens in the case of a homogeneous nonlinearity
  (classically of the form $\lambda|u|^{p}u$), replacing $u$ with $k  u$
  ($k\in \C\setminus\{0\}$)
  in \eqref{LSE}  has only little effect, since we have
  \begin{equation*}
    i\partial_t (k u) +\Delta (k u) =\lambda k u
    \ln\(|k u|^2\) - \lambda (\ln |k|^2 )k u \, .
  \end{equation*}
The scaling factor thus corresponds to a purely time-dependent gauge
transform:
\begin{equation*}
  k u (\bx,t) e^{-it \lambda \ln |k|^2}
\end{equation*}
solves \eqref{LSE} (with initial datum $k u_0$). In
particular, the size of the initial datum does not influence the
dynamics of the solution. In spite of this property which is
reminiscent of linear equations, nonlinear effects are stronger in
\eqref{LSE} than in, say, cubic Schr\"odinger equations in several
respects. For $\Omega=\R^d$, it was established in \cite{Caz83} that in the case
$\lambda<0$, no solution is dispersive (not even for small data, in
view of the above remark), while if $\lambda>0$, the results from
\cite{CaGa-p} show that every solution disperses, at a faster rate
than for the linear equation.
\smallbreak

In view of the gauge invariance of the nonlinearity, for
$\Omega=\mathbb R^d$, \eqref{LSE}
enjoys the standard Galilean invariance: if $u (\bx,t)$ solves
\eqref{LSE}, then, for any $\mathbf v\in \R^d$, so does
\begin{equation*}
  u(\bx -2\mathbf v t,t)e^{i\mathbf v\cdot \bx -i|\mathbf v|^2t}.
\end{equation*}
A remarkable feature of \eqref{LSE} is that it possesses a large set
of explicit solutions. In the case $\Omega=\R^d$: if $u_0$ is
Gaussian, $u(\cdot,t)$ is Gaussian for all time, and solving
\eqref{LSE} amounts to solving ordinary differential equations
\cite{BiMy76}. For simplicity of notation, we take the one-dimensional case as an example. If the initial data in \eqref{LSE} with $\Omega={\mathbb R}$
is taken as
\begin{equation*}
  u_0(x) = b_0 e^{-\frac{a_0}{2} x^2+ivx}, \qquad x\in{\mathbb R},
\end{equation*}
where $a_0,b_0\in \C$ and $v\in\mathbb{R}$ are given constants
satisfying $\alpha_0:=\mathrm{Re}\, a_0>0$ with ${\rm Re}\, f$ denoting the real part of $f$,
then the solution of \eqref{LSE} is given by \cite{Ar16,CaGa-p}
\be\label{Gaus}
  u(x,t) = \frac{b_0}{\sqrt{r(t)}}e^{i(vx-v^2t)+Y(x-2vt,t)},\qquad x\in{\mathbb R}, \quad t\ge0,
\ee
with
\be
Y(x,t)=-i\phi(t)-\alpha_0
  \frac{x^2}{2r(t)^2}+i\frac{\dot r(t)}{r(t)}\frac{x^2}{4},
  \qquad x\in{\mathbb R}, \quad t\ge0,
\ee
where
$\phi:=\phi(t)\in\R$ and $r:=r(t)>0$ solve the ODEs \cite{Ar16,CaGa-p}
\be\label{ode}
\begin{split}
&  \dot \phi= \frac{\alpha_0}{r^2} +\lambda \ln|b_0|^2-\lambda \ln
  r,\quad \phi(0)=0,\\
& \ddot r = \frac{4\alpha_0^2}{r^3}+\frac{4\lambda \alpha_0}{r},\quad
  r(0)=1,\ \dot r(0)= -2\,\mathrm{Im}\,a_0.
\end{split}
\ee
In the case $\lambda<0$, the function $r$ is  (time)
periodic (in agreement with the absence of dispersive
effects). In particular, if $a_0=-\lambda>0$, it follows from \eqref{ode} that
$r(t)\equiv 1$ and $\phi(t)= \phi_0t$ with $\phi_0=\lambda\left[\ln(|b_0|^2)-1\right]$, which generates the uniformly {\bf moving Gausson} as \cite{Ar16,CaGa-p}
\be\label{Gausson}
u(x,t)=b_0 e^{\frac{\lambda}{2}(x-2vt)^2+i(vx-(\phi_0+v^2)t)}, \qquad x\in{\mathbb R}, \quad t\ge0.
\ee
As a very special case with $b_0=e^{1/2}$ and $v=0$ such that $\phi_0=0$, one can get the {\bf static
Gausson} as
\be
u(x,t)=e^{1/2}e^{\lambda|x|^2/2}, \qquad x\in{\mathbb R}, \quad t\ge0.
\ee
This special solution is orbitally stable \cite{Caz83,CaLi82}. On the other hand, in the case $\lambda>0$, it is proven in \cite{CaGa-p}
that for general initial data (not necessarily Gaussian), there exists a
universal dynamics. For extensions to higher dimensions, we refer to \cite{Ar16,CaGa-p} and references therein.
Therefore, \eqref{LSE} possesses several specific features, which make
it quite different from the nonlinear Schr\"odinger equation.

Different numerical methods have been proposed and analyzed for the nonlinear Schr\"odinger equation with smooth nonlinearity (e.g. cubic nonlinearity) in the literature,
such as the finite difference methods \cite{bao2012, bao2013}, finite
element methods \cite{akrivis, FEM} and the time-splitting pseudospectral methods \cite{bao2003, taha1984}. However, they cannot be applied to the
LogSE \eqref{LSE} directly due to the blow-up of the logarithmic nonlinearity, i.e. $\ln \rho\to -\infty$ when $\rho\rightarrow 0^+$.
The main aim of this paper is to present a regularized finite difference method for the LogSE \eqref{LSE} by introducing a proper
regularized logarithmic Schr\"odinger equation (RLogSE) and
then discretizing the RLogSE via a semi-implicit finite difference method.
Error estimates will be established between the solutions of
LogSE and RLogSE as well as their numerical approximations.

The rest of the paper is organized as follows. In
Section~\ref{sec:regul}, we propose a regularized version of
\eqref{LSE} with a small regularization parameter $0<\ep\ll1$, and analyze its properties, as well as the convergence of
its solution to the solution of \eqref{LSE}. In
Section~\ref{sec:scheme}, we introduce a semi-implicit finite difference method for discretizing the regularized logarithmic Schr\"odinger equation, and prove an error estimate, in which the dependence of the constants with respect to the regularization parameter $\ep$ is tracked very explicitly. Finally, numerical results are provided in Section~\ref{sec:illust} to confirm our error bounds and to demonstrate the efficiency and accuracy of the proposed numerical method.

Throughout the paper, we use $H^m(\Omega)$ and
$\|\cdot\|_{H^m(\Omega)}$ to denote the standard Sobolev spaces and
their norms, respectively. In particular, the norm and inner product
of $L^2(\Omega)=H^0(\Omega)$ are denoted by $\|\cdot\|_{L^2(\Omega)}$
and $(\cdot, \cdot)$, respectively. Moreover, we adopt $A\lesssim B$ to
mean that there exists a generic constant $C>0$ independent of the regularization parameter $\ep$, the time step
$\tau$ and the mesh size $h$ such that $A\le C\,B$, and $\lesssim_c$ means the constant $C$ depends on $c$.

\section{A regularized logarithmic Schr\"odinger equation}
\label{sec:regul}
It turns out that a direct simulation of the solution of \eqref{LSE}
is very delicate, due to the singularity of the logarithm at the
origin, as discussed in the introduction. Instead of working directly
with \eqref{LSE}, we shall
consider the following regularized logarithmic Schr\"odinger equation (RLogSE)
with a samll regularized parameter
$0<\ep\ll1 $ as
\be\label{RLSE}
\left\{
\begin{aligned}
&i\p_t u^\ep(\bx,t)+\Delta u^\ep(\bx,t)=\lambda
u^\ep(\bx,t)\,\ln\(\ep+|u^\ep(\bx,t)|\)^2,\quad \bx\in \Omega, \quad
t>0,\\
&u^\ep(\bx,0)=u_0(\bx),\quad \bx\in \overline{\Omega}.
\end{aligned}
\right.
\ee

\subsection{Conserved quantities}

For the RLogSE \eqref{RLSE}, it can be similarly deduced
that the mass, momentum,  and energy are conserved.
\begin{proposition}\label{prop:conserv}
The  mass, momentum,  and `regularized' energy are formally conserved for the RLogSE \eqref{RLSE}:
\be\label{regen}
\begin{split}
&M^\ep(t):=\int_\Omega|u^\ep(\bx,t)|^2d\bx\equiv \int_\Omega|u_0(\bx)|^2d\bx=
M(0),\\
&P^\ep(t):= \mathrm{Im} \int_\Omega \overline{u^\ep}(\bx,t)\nabla
  u^\ep(\bx,t)d\bx\equiv \mathrm{Im} \int_\Omega \overline{u_0}(\bx)\nabla
  u_0(\bx)d\bx=P(0), \quad t\ge0,\\
&E^{\ep}(t):=\int_{\Omega}\big[|\nabla u^\ep(\bx,t)|^2+\lambda
F_\ep(|u^\ep(\bx,t)|^2)\big](\bx,t)d\bx\\
&\qquad\,\,\,\equiv\int_{\Omega}\big[|\nabla u_0(\bx)|^2+\lambda F_\ep(|u_0(\bx)|^2)\big]d\bx=E^\ep(0),
\end{split}
\ee
where
\be\label{Renerg}
\begin{aligned}
F_\ep(\rho)&=\int_0^\rho \ln(\ep+\sqrt{s})^2ds\\
&=\rho\ln\left(\ep+\sqrt{\rho}\right)^2-\rho+2 \ep \sqrt{\rho}- \ep^2
\ln\left(1+\sqrt{\rho}/\ep\right)^2, \qquad \rho\ge0.
\end{aligned}
\ee
\end{proposition}
\begin{proof}
The conservation for mass and momentum is standard, and relies on the fact that the
right hand side of \eqref{RLSE} involves $u^\ep$ multiplied
by a \emph{real} number. For the energy $E^\ep(t)$, we compute
\begin{align*}
\fl{d}{dt}E^\ep(t)&=2\,\mathrm{Re}\int_\Omega \left[\nabla u^\ep\cdot \nabla \p_t \overline{u^\ep}+\lambda  u^\ep\p_t \overline{u^\ep}\ln (\ep+|u^\ep|)^2-\lambda  u^\ep\p_t \overline{u^\ep}\right](\bx,t)d\bx\\
&\quad+2\lambda\int_\Omega \p_t|u^\ep|\Big[\ep+\fl{|u^\ep|^2-\ep^2}{\ep+|u^\ep|}\Big](\bx,t)d\bx\\
&=2\,\mathrm{Re}\int_\Omega \left[\p_t \overline{u^\ep}\left(-\Delta u^\ep+\lambda u^\ep\ln (\ep+|u^\ep|)^2\right)\right](\bx,t)d\bx\\
&=2\,\mathrm{Re}\int_\Omega i|\p_t u^\ep|^2(\bx,t)d\bx=0,\quad t\ge0,
\end{align*}
which completes the proof.
\end{proof}

Note however that since the above `regularized' energy involves $L^1$-norm
of $u^\ep$ for any $\ep>0$, $E^\ep$ is obviously well-defined for $u_0\in H^1(\Omega)$ when $\Omega$ has finite measure, but not when $\Omega=\R^d$. This aspect is discussed more into details
in Subsections~\ref{sec:whole} and \ref{sec:cvenergy}.

\subsection{The Cauchy problem}

For $\alpha>0$ and $\Omega=\R^d$, denote by $L^2_\alpha$ the weighted $L^2$
space
\[L^2_\alpha:=\{v\in L^2(\mathbb{R}^d), \quad \bx\longmapsto \langle
  \bx \rangle^\alpha v(\bx)\in L^2(\mathbb{R}^d)\},\]
where $\langle \bx \rangle :=\sqrt{1+|\bx|^2}$, with norm
$$\|v\|_{L^2_\alpha}:=\|\langle \bx\rangle^\alpha
v(\bx)\|_{L^2(\mathbb{R}^d)}.$$
In the case where $\Omega$ is bounded, we simply set
$L^2_\alpha=L^2(\Omega)$.
Regarding the Cauchy problems \eqref{LSE} and \eqref{RLSE}, we have the following result.

\bigskip

\begin{theorem}\label{theo:cauchy}
Let $\lambda\in \R$ and $\ep>0$. Consider \eqref{LSE} and \eqref{RLSE}
on $\Omega=\R^d$, or bounded $\Omega$ with homogeneous Dirichlet or periodic
boundary condition.
Consider an initial datum
$u_0\in H^1_0(\Omega)\cap L^2_\alpha$, for some $0<\alpha\le 1$.
\begin{itemize}
\item There exists a unique, global  solution $u\in L^\infty_{\rm
    loc}(\R; H^1_0(\Omega)\cap L^2_\alpha)$ to \eqref{LSE}, and a unique, global
  solution $u^\ep\in L^\infty_{\rm
    loc}(\R; H^1_0(\Omega)\cap L^2_\alpha)$ to \eqref{RLSE}.
\item If in addition $u_0\in H^2(\Omega)$, then $u,u^\ep \in
L^\infty_{\rm loc}(\R;H^2(\Omega))$.
\item In the case $\Omega=\R^d$, if in addition $u_0\in H^2\cap
  L_2^2$, then  $u,u^\ep \in L^\infty_{\rm loc}(\R;H^2\cap L_2^2)$.
\end{itemize}
\end{theorem}

\medskip

\begin{proof}
This result can be proved by using more or less directly the
arguments invoked in \cite{CaGa-p}. First, for fixed $\ep>0$, the
nonlinearity in \eqref{RLSE} is locally Lipschitz, and grows more
slowly than any power for large $|u^\ep|$. Therefore, the standard Cauchy
theory for nonlinear Schr\"odinger equations applies (see in
particular \cite[Corollary~3.3.11 and Theorem~3.4.1]{CazCourant}), and so
if $u_0\in H^1_0(\Omega)$, then \eqref{RLSE} has a unique solution
$u^\ep\in L^\infty_{\rm loc}(\R;H^1_0(\Omega))$. Higher Sobolev
regularity is propagated, with controls depending on $\ep$ in
general.
\smallbreak

A solution $u$ of \eqref{LSE} can be obtained by compactness arguments, by letting
$\ep\to 0$ in \eqref{RLSE}, provided that we have suitable bounds
independent of $\ep>0$.
We have
\begin{equation*}
  i\partial_t \nabla u^\ep + \Delta \nabla u^\ep=2\lambda
  \ln\(\ep+|u^\ep|\)\nabla u^\ep + 2\lambda
  \frac{u^\ep}{\ep+|u^\ep|}  \nabla |u^\ep|.
\end{equation*}
The standard energy estimate (multiply the above equation by $\nabla \overline{u^\ep}$, integrate over $\Omega$ and take the imaginary part)
yields, when $\Omega=\R^d$ or when periodic boundary conditions are considered,
\begin{equation*}
  \frac{1}{2}\frac{d}{dt}\|\nabla u^\ep\|_{L^2(\Omega)}^2 \le 2|\lambda|
\int_\Omega  \frac{|u^\ep|}{\ep+|u^\ep|} \left|\nabla
  |u^\ep|\right|\lvert\nabla u^\ep\rvert d\bx\le 2|\lambda \|\nabla u^\ep\|_{L^2(\Omega)}^2 .
\end{equation*}
Gronwall lemma yields a bound for $u^\ep$ in $L^\infty(0,T;H^1(\Omega))$,
uniformly in $\ep>0$, for any given $T>0$.
Indeed, the above estimate uses the property
\begin{equation*}
  \mathrm{Im}\int_\Omega \nabla \overline{u^\ep}\cdot \Delta \nabla u^\ep\,d\bx=0,
\end{equation*}
which needs not be true when $\Omega$ is bounded and $u^\ep$ satisfies
homogeneous Dirichlet boundary conditions. In that case, we use the
conservation of the energy $E^\ep$ (Proposition~\ref{prop:conserv}),
and write
\begin{align*}
  \|\nabla u^\ep (t)\|_{L^2(\Omega)}^2 &\le E^\ep(u_0)
    +2|\lambda|\int_\Omega|u^\ep(\bx,t)|^2\left\lvert
      \ln\(\ep+|u^\ep(\bx,t)|\)\right\rvert d\bx \\
&\quad+2\ep|\lambda|
    \|u^\ep(t)\|_{L^1(\Omega)}+2|\lambda
    |\ep^2\int_\Omega \left|\ln
      \left(1+|u^\ep(\bx,t)|/\ep\right)\right| d\bx\\
&\lesssim 1 + \ep|\Omega|^{1/2}\|u^\ep(t)\|_{L^2(\Omega)} +
  \int_\Omega|u^\ep(\bx,t)|^2\left\lvert
      \ln\(\ep+|u^\ep(\bx,t)|\)\right\rvert d\bx \\
&\lesssim 1 +
  \int_\Omega|u^\ep(\bx,t)|^2\left\lvert
      \ln\(\ep+|u^\ep(\bx,t)|\)\right\rvert d\bx,\quad t\ge0,
\end{align*}
where we have used Cauchy-Schwarz inequality and the conservation of
the mass $M^\ep(t)$. Writing, for $0<\eta\ll 1$,
\begin{align*}
  &\hspace{-3mm}\int_\Omega|u^\ep|^2\left\lvert
      \ln\(\ep+|u^\ep|\)\right\rvert d\bx \\
      & \lesssim \int_{\ep+|u^\ep|>1}
  |u^\ep|^2\(\ep+|u^\ep|\)^\eta d\bx+ \int_{\ep+|u^\ep|<1}
  |u^\ep|^2\(\ep+|u^\ep|\)^{-\eta}d\bx \\
&\lesssim \|u^\ep\|_{L^2(\Omega)} + \|u^\ep\|_{L^{2+\eta}(\Omega)}^{2+\eta} +
  \|u^\ep\|_{L^{2-\eta}(\Omega)}^{2-\eta}\lesssim 1 + \|\nabla u^\ep\|_{L^2(\Omega)}^{d\eta/2},
\end{align*}
where we have used the interpolation inequality (see e.g. \cite{Nir59})
\begin{equation*}
  \|u\|_{L^p(\Omega)}\lesssim \|u\|_{L^2(\Omega)}^{1-\alpha}\|\nabla
  u\|_{L^2(\Omega)}^{\alpha}+\|u\|_{L^2(\Omega)},\quad p=\frac{2d}{d-2\alpha},\ 0\le \alpha<1,
\end{equation*}
we obtain again that $u^\ep$ is bounded in $L^\infty(0,T;H^1(\Omega))$,
uniformly in $\ep>0$, for any given $T>0$.
\smallbreak

In the case where $\Omega$ is bounded, compactness arguments show that
$u^\ep$ converges to a solution $u$ to \eqref{LSE}; see
\cite{CazCourant,CaHa80}. When $\Omega=\R^d$, compactness in space is
provided by multiplying \eqref{RLSE} with~$ \langle  \bx \rangle^{2\alpha}
\overline {u^\ep}$ and integrating in space:
$$
\frac d{dt} \|u^\ep\|_{L^2_\alpha}^2 = 4\alpha\,\mathrm{Im}\int  \frac{\bx
  \cdot \nabla  u^{\ep}  }{  \langle \bx \rangle^{2-2\alpha}} \,
\overline {u^\ep}(t)  \, d\bx \lesssim
\|\<\bx\>^{2\alpha-1}u^\ep\|_{L^2(\Omega)}\|\nabla u^\ep\|_{L^2(\Omega)},
$$
where we have used Cauchy-Schwarz inequality. Recalling that
$0<\alpha\le 1$,
\begin{equation*}
  \|\<\bx\>^{2\alpha-1}u^\ep\|_{L^2(\Omega)}\le
  \|\<\bx\>^{\alpha}u^\ep\|_{L^2(\Omega)}=\|u^\ep\|_{L^2_\alpha},
\end{equation*}
and we obtain a bound for $u^\ep$ in $L^\infty(0,T;H^1(\Omega)\cap
L^2_\alpha)$ which is uniform in $\ep$. Uniqueness of such a solution
for \eqref{LSE} follows from the arguments of \cite{CaHa80}, involving
a specific algebraic inequality, generalized in
Lemma~\ref{pre} below. Note that at this stage, we know that $u^\ep$
converges to $u$ by compactness arguments, so we have no convergence
estimate. Such estimates are established in
Subsection~\ref{sec:cvmodel}.
 \smallbreak

To prove the propagation of the $H^2$ regularity, we note that
differentiating twice the nonlinearity in \eqref{RLSE} makes it
unrealistic to expect direct bounds which are uniform in $\ep$. To
overcome this difficulty, the argument proposed in \cite{CaGa-p}
relies on Kato's idea: instead of differentiating the equation twice in
space, differentiate it once in time, and use the equation to infer
$H^2$ regularity. This yields the second part of the theorem.
\smallbreak

To establish the last part of the theorem, we prove that $u\in
L_{\mathrm{loc}}^\infty(\mathbb{R}; L^2_2)$ and
the same approach applies to $u^\ep$.
It follows from \eqref{LSE} that
\begin{align}
\fl{d}{dt}\|u(t)\|_{L^2_2}^2&=-2\,\mathrm{Im}
\int_{\mathbb{R}^d}\langle \bx\rangle^4\overline{u(\bx,t)}\Delta u(\bx,t)d\bx\nn\\
&=8\,\mathrm{Im}\int_{\mathbb{R}^d}\langle \bx\rangle^2\overline{u(\bx,t)}\,\bx\cdot\nabla u (\bx,t)d\bx
\le 8\,\|u(t)\|_{L^2_2} \|\bx\cdot \nabla u(t)\|_{L^2(\mathbb{R}^d)}.\label{2ndm}
\end{align}
By Cauchy-Schwarz inequality and integration by parts, we have
\begin{align*}
&\hspace{-3mm}\|\bx\cdot \nabla u(t)\|_{L^2(\mathbb{R}^d)}^2\\
&\le \sum\limits_{j=1}^d\sum\limits_{k=1}^d \int_{\mathbb{R}^d}
x_j^2\fl{\p u(\bx,t)}{\p x_k}\fl{\p \overline{u(\bx,t)}}{\p x_k}d\bx\\
&=-2\int_{\mathbb{R}^d}\overline{u(\bx,t)}\,\bx\cdot \nabla u(\bx,t)d\bx-\int_{\mathbb{R}^d}|\bx|^2\overline{u(\bx,t)}\Delta u(\bx,t)d\bx\\
&\le \fl{1}{2}\|\bx\cdot \nabla u(t)\|_{L^2(\mathbb{R}^d)}^2+ 2\|u(t)\|_{L^2(\mathbb{R}^d)}^2
+\fl{1}{2}\|u(t)\|_{L^2_2}^2+\fl{1}{2}
\|\Delta u(t)\|_{L^2(\mathbb{R}^d)}^2,
\end{align*}
which yields directly that
\[\|\bx\cdot \nabla u(t)\|_{L^2(\mathbb{R}^d)}\le
  2\|u(t)\|_{L^2(\mathbb{R}^d)}+\|u(t)\|_{L^2_2}+\|\Delta
  u(t)\|_{L^2(\mathbb{R}^d)}.\]
This together with \eqref{2ndm} gives that
\begin{equation*}
\fl{d}{dt}\|u(t)\|_{L^2_2}\le 4\|\bx\cdot \nabla u(t)\|_{L^2(\mathbb{R}^d)}
\le 4\|u(t)\|_{L^2_2}+8\|u(t)\|_{L^2(\mathbb{R}^d)}+4\|\Delta u(t)\|_{L^2(\mathbb{R}^d)}.
\end{equation*}
Since we already know that  $u\in L_{\mathrm{loc}}^\infty(\mathbb{R};
H^2(\mathbb{R}^d))$, Gronwall lemma completes the proof.
\end{proof}

\begin{remark}\label{rem:Hk}
  We emphasize that if $u_0\in H^k(\R^d)$, $k\ge
3$, we cannot guarantee in general that this higher regularity is
propagated in \eqref{LSE}, due to the singularities stemming from the
logarithm. Still, this property is fulfilled in the case where
$u_0$ is Gaussian, since then $u$ remains Gaussian for all
time. However, our numerical tests, in the case where
  the initial datum is chosen as  the dark soliton of the cubic Schr\"odinger
  equation  multiplied by a Gaussian, suggest that even the $H^3$ regularity is not propagated in general.
\end{remark}
\subsection{Convergence of the regularized model}
\label{sec:cvmodel}
In this subsection, we show the approximation property of the
regularized model \eqref{RLSE} to \eqref{LSE}.

\subsubsection{A general estimate}
We prove:
\begin{lemma}\label{thmcon}
Suppose the equation is set on $\Omega$, where $\Omega=\mathbb{R}^d$,
or $\Omega\subset \mathbb{R}^d$ is a bounded domain with homogeneous
Dirichlet or periodic boundary condition, then we have
the general estimate:
\be\label{gene}
\fl{d}{dt}\|u^\ep(t)-u(t)\|_{L^2(\Omega)}^2\le 4|\lambda|\left(\|u^\ep(t)-u(t)\|_{L^2(\Omega)}^2+
\ep\|u^\ep(t)-u(t)\|_{L^1(\Omega)}\right).
\ee
\end{lemma}

Before giving the proof of Lemma \ref{thmcon}, we introduce the
following lemma, which is a variant of
\cite[Lemma~9.3.5]{CazCourant}, established initially in \cite[Lemme~1.1.1]{CaHa80}.
\begin{lemma}\label{pre}
Let $\ep\ge 0$ and denote
$f_\ep(z)=z\ln(\ep+|z|)$, then we have
\[\left|\mathrm{Im}\left(\(f_\ep(z_1)-f_\ep(z_2)\)
(\overline{z_1}-\overline{z_2})\right)\right|\le
  |z_1-z_2|^2,  \quad z_1, z_2\in\mathbb{C}. \]
\end{lemma}
\begin{proof}
Notice that
\[\mathrm{Im}\left[\(f_\ep(z_1)-f_\ep(z_2)\)(\overline{z_1}-\overline{z_2})\right]=\fl{1}{2}
\left[\ln(\ep+|z_1|)-\ln(\ep+|z_2|)\right]\mathrm{Im}(\overline{z_1}z_2-z_1\overline{z_2}).\]
Supposing, for example, $0<|z_2|\le|z_1|$, we can obtain that
\[\left|\ln(\ep+|z_1|)-\ln(\ep+|z_2|)\right|=\ln\left(1+\fl{|z_1|-|z_2|}{\ep+|z_2|}\right)\le\fl{|z_1|-|z_2|}
{\ep+|z_2|}\le \fl{|z_1-z_2|}{|z_2|},\]
and
\[\left|\mathrm{Im}(\overline{z_1}z_2-z_1\overline{z_2})\right|=
\left|z_2(\overline{z_1}-\overline{z_2})+\overline{z_2}(z_2-z_1))\right|\le 2|z_2|\,|z_1-z_2|.\]
Otherwise the result follows by exchanging $z_1$ and $z_2$.
\end{proof}

\begin{proof}(Proof of Lemma \ref{thmcon})
Subtracting \eqref{LSE} from \eqref{RLSE}, we see that the error function $e^\ep:=u^\ep-u$ satisfies
\[
i\p_t e^\ep+\Delta e^\ep=\lambda \left[u^\ep\ln(\ep+|u^\ep|)^2-u\ln(|u|^2)\right].
\]
Multiplying the error equation by $\overline{e^\ep(t)}$, integrating in space and taking the imaginary parts, we can get by using Lemma \ref{pre} that
\begin{align*}
\fl{1}{2}\fl{d}{dt}\|e^\ep(t)\|_{L^2(\Omega)}^2&=2\lambda\, \mathrm{Im} \int_{\Omega}
\left[u^\ep\ln(\ep+|u^\ep|)-u\ln(|u|)\right](\overline{u^\ep}-\overline{u})(\bx,t)d\bx\\
&\le 2 |\lambda|\|e^\ep(t)\|_{L^2(\Omega)}^2+2|\lambda|\left|\int_{\Omega}\overline{e^\ep}u
\left[\ln(\ep+|u|)-\ln(|u|)\right](\bx,t)d\bx\right|\\
&\le 2|\lambda|\|e^\ep(t)\|_{L^2(\Omega)}^2+2\ep|\lambda|\|e^\ep(t)\|_{L^1(\Omega)},
\end{align*}
where we have used the general estimate $0\le\ln (1+|x|)\le |x|$.
\end{proof}
\subsubsection{Convergence for bounded domain}
If $\Omega$ has finite measure, then we can have the following convergence behavior.
\begin{proposition}\label{prop1}
Assume that $\Omega$ has finite measure, and let $u_0\in
H^2(\Omega)$.
For any $T>0$, we have
\be
\label{Conv_RLSE}
\|u^\ep-u\|_{L^\infty(0,T; L^2(\Omega))}\le C_1\ep,\quad
\|u^\ep-u\|_{L^\infty(0,T; H^1(\Omega))}\le C_2\ep^{1/2},
\ee
where $C_1$ depends on $|\lambda|$, $T$, $|\Omega|$ and $C_2$ depends on $|\lambda|$, $T$, $|\Omega|$ and $\|u_0\|_{H^2(\Omega)}$.
\end{proposition}
\begin{proof}
Note that $\|e^\ep(t)\|_{L^1(\Omega)}\le |\Omega|^{1/2}\|e^\ep(t)\|_{L^2(\Omega)}$, then it follows from \eqref{gene} that
\[\fl{d}{dt}\|e^\ep(t)\|_{L^2(\Omega)}\le 2|\lambda|\|e^\ep(t)\|_{L^2(\Omega)}+2\ep|\lambda| |\Omega|^{1/2}.\]
Applying Gronwall's inequality, we immediately get that
\[\|e^\ep(t)\|_{L^2(\Omega)}\le \left(\|e^\ep(0)\|_{L^2(\Omega)}+\ep |\Omega|^{1/2}\right)e^{2|\lambda|t}=\ep |\Omega|^{1/2}e^{2|\lambda|t}.\]
The convergence rate in $H^1$ follows from the property
$u^\ep, u\in L_{\mathrm{loc}}^\infty(\mathbb{R}; H^2(\Omega))$
and the Gagliardo-Nirenberg inequality \cite{leoni2017},
\[\|\nabla v\|_{L^2(\Omega)}\lesssim \|v\|_{L^2(\Omega)}^{1/2}\|\Delta v\|_{L^2(\Omega)}^{1/2},\]
which completes the proof.
\end{proof}
\begin{remark}
  The weaker rate in the $H^1$ estimate is due to the fact that
  Lemma~\ref{thmcon} is not easily adapted to $H^1$ estimates, because
  of the presence of the logarithm. Differentiating \eqref{LSE} and
  \eqref{RLSE} makes it hard  to obtain the analogue in
  Lemma~\ref{thmcon}. This is why we bypass this difficulty by
  invoking boundedness in $H^2$ and interpolating with the  error bound
  at the $L^2$ level. If we have $u^\ep$, $u\in L^\infty_{\mathrm{loc}}(\mathbb{R};
H^k(\Omega))$ for $k>2$, then the convergence rate in
$H^1(\Omega)$ can be improved as
\[\| e^\ep\|_{L^\infty(0,T; H^1(\Omega))}\lesssim \ep^{\fl{k-1}{k}},\]
by using the inequality (see e.g. \cite{Nir59}):
\[\| v\|_{H^1(\Omega)}\lesssim \|v\|_{L^2(\Omega)}^{1-1/k}\|
  v\|_{H^k(\Omega)}^{1/k}.\]
\end{remark}

\subsubsection{Convergence for the whole space}\label{sec:whole}
In order to prove the convergence rate of the regularized model \eqref{RLSE} to \eqref{LSE} for the whole space, we need the following lemma.
\begin{lemma}\label{lem:localisation}
For $d=1, 2, 3$, if $v\in L^2(\mathbb{R}^d)\cap L^2_2$, then we have
\be\label{mod}
\|v\|_{L^1(\mathbb{R}^d)}\le C \|v\|_{L^2(\mathbb{R}^d)}^{1-d/4}\|v\|_{L^2_2}^{d/4},
\ee
where $C>0$ depends on $d$.
\end{lemma}
\begin{proof}
Applying the Cauchy-Schwarz inequality, we can get for fixed $r>0$,
\begin{align*}
\|v\|_{L^1(\mathbb{R}^d)}&=\int_{|\bx|\le r} |v(\bx)| d\bx
                           +\int_{|\bx|\ge r}
                           \fl{|\bx|^{2}|v(\bx)|}{|\bx|^{2}}d\bx\\
&\lesssim  r^{d/2}\left(\int_{|\bx|\le r} |v(\bx)|^2 d\bx\right)^{\fl{1}{2}} +\left(\int_{|\bx|\ge r} |\bx|^4|v(\bx)|^2d\bx\right)^{\fl{1}{2}}
\left(\int_{|\bx|\ge r} \fl{1}{|\bx|^4}d\bx\right)^{\fl{1}{2}}\\
&\lesssim r^{d/2}\|v\|_{L^2(\mathbb{R}^d)}+r^{d/2-2}
\|v\|_{L^2_2}.
\end{align*}
Then \eqref{mod} can be obtained by setting
$r=\(\|v\|_{L^2_2}/\|v\|_{L^2(\mathbb{R}^d)}\)^{1/2}$.
\end{proof}

\bigskip

\begin{proposition}
Assume that $\Omega=\R^d$, $1\le d\le 3$, and let $u_0\in
H^2(\R^d)\cap L^2_2$.
For any $T>0$, we have
\[\|u^\ep-u\|_{L^\infty(0,T; L^2(\mathbb{R}^d))}\le C_1\ep^{\fl{4}{4+d}},\quad
\|u^\ep-u\|_{L^\infty(0,T; H^1(\mathbb{R}^d)))}\le C_2\ep^{\fl{2}{4+d}},\]
where $C_1$ depends on $d$, $|\lambda|$, $T$, $\|u_0\|_{L^2_2}$ and $C_2$ depends on additional $\|u_0\|_{H^2(\mathbb{R}^d)}$.
\end{proposition}
\begin{proof}
Applying \eqref{mod} and the Young's inequality, we deduce that
\begin{equation*}
\ep\|e^{\ep}(t)\|_{L^1(\mathbb{R}^d)}\le  \ep C_d\|e^\ep(t)\|_{L^2(\mathbb{R}^d)}^{1-d/4}\|e^\ep(t)\|_{L^2_2}^{d/4}
\le C_d\left(\|e^\ep(t)\|_{L^2(\mathbb{R}^d)}^2+\ep^{\fl{8}{4+d}}\|e^\ep(t)\|_{L^2_2}^{\fl{2d}{4+d}}\right),
\end{equation*}
which together with \eqref{gene} gives that
\[\fl{d}{dt}\|e^\ep(t)\|_{L^2(\mathbb{R}^d)}^2\le 4|\lambda|(1+C_d)\|e^\ep(t)\|_{L^2(\mathbb{R}^d)}^2+4C_d|\lambda|
  \ep^{\fl{8}{4+d}}\|e^\ep(t)\|_{L^2_2}^{\fl{2d}{4+d}}.\]
 Gronwall lemma yields
\[\|e^\ep(t)\|_{L^2(\mathbb{R}^d)}\le \ep^{\fl{4}{4+d}}\|e^\ep(t)\|_{L^2_2}^{\fl{d}{4+d}}e^{tC_{d, |\lambda|}}.\]
The proposition follows by recalling that
$u^\ep$, $u\in L^\infty_{\mathrm{loc}}(\mathbb{R};
H^2(\mathbb{R}^d)\cap L^2_2)$.
\end{proof}

\begin{remark}
If we have $u^\ep, u\in L^\infty_{\mathrm{loc}}(\mathbb{R}; L^2_m)$
for $m>2$, then by applying the inequality
\[\ep\|v\|_{L^1(\mathbb{R}^d)}\lesssim\ep
  \|v\|_{L^2(\mathbb{R}^d)}^{1-\fl{d}{2m}}\|v\|_{L^2_m}^{\fl{d}{2m}}\lesssim
  \|v\|_{L^2(\mathbb{R}^d)}^2+\ep^{\fl{4m}{2m+d}}\|v\|_{L^2_m}^{\fl{2d}{2m+d}},\]
which can be proved like above,
the convergence rate can be improved as
\[\|u^\ep-u\|_{L^\infty(0,T; L^2(\mathbb{R}^d))}\lesssim \ep^{\fl{2m}{2m+d}}.\]
\end{remark}

\begin{remark}\label{rem:Hk2}
If in addition $u^\ep, u\in L^\infty_{\mathrm{loc}}(\mathbb{R}; H^s(\mathbb{R}^d))$ for $s>2$, then the convergence rate in $H^1(\mathbb{R}^d)$ can be improved as
\[\| e^\ep\|_{L^\infty(0,T; H^1(\mathbb{R}^d)))}\le C\ep^{\fl{2m}{2m+d}\fl{s-1}{s}},\]
by using the Gagliardo-Nirenberg inequality:
\[\|\nabla v\|_{L^2(\mathbb{R}^d)}\le C\|v\|_{L^2(\mathbb{R}^d)}^{1-1/s}\|\nabla^s v\|_{L^2(\mathbb{R}^d)}^{1/s}.\]
\end{remark}
The previous two remarks apply typically in the case of Gaussian
initial data.
\subsection{Convergence of the energy}
\label{sec:cvenergy}
In this subsection we will show the convergence of the energy
$E^\ep(u_0)\to E(u_0)$.
\begin{proposition}\label{energyc}
For $u_0\in H^1(\Omega)\cap L^1(\Omega)$,  the energy $E^\ep(u_0)$
converges to $E(u_0)$ with
$$|E^\ep(u_0)-E(u_0)|\le 4\,\ep|\lambda| \|u_0\|_{L^1(\Omega)}.$$
\end{proposition}
\begin{proof} It can be deduced from the definition that
\begin{align*}
\left|E^\ep(u_0)-E(u_0)\right|&=2|\lambda|\left|\ep \|u_0\|_{L^1(\Omega)}+\int_{\Omega} |u_0(\bx)|^2 \left[\ln(\ep+|u_0(\bx)|-\ln(|u_0(\bx)|)\right]d\bx\right.\\
&\quad\left.-\ep^2
\int_{\Omega}\ln\left(1+|u_0(\bx)|/\ep\right)d\bx\right|\\
&\le 4\,\ep|\lambda| \|u_0\|_{L^1(\Omega)},
\end{align*}
which completes the proof.
\end{proof}

\begin{remark}
  If $\Omega$ is bounded, then $H^1(\Omega)\subseteq L^1(\Omega)$. If
$\Omega=\R^d$, then Lemma~\ref{lem:localisation} (and its natural
generalizations) shows that $H^1(\R)\cap L^2_1\subseteq L^1(\R)$, and if
$d=2,3$, $ H^1(\R^d)\cap L^2_2\subseteq L^1(\R^d)$.
\end{remark}

\begin{remark}
This regularization is reminiscent of the one considered in
\cite{CaGa-p} in order to prove (by compactness arguments) that
\eqref{LSE} has a solution,
\begin{equation}
\label{RLogSE_2}
  i\p_t u^\ep(\bx,t)+\Delta u^\ep(\bx,t)=\lambda
  u^\ep(\bx,t)\,\ln\(\ep+|u^\ep(\bx,t)|^2\),\quad \bx\in \Omega, \quad
t>0.
\end{equation}
With that regularization, it is easy to adapt the error estimates
established above for \eqref{RLSE}. Essentially, $\ep$ must be
replaced by $\sqrt\ep$ (in Lemma~\ref{thmcon}, and hence in its corollaries).
\end{remark}
\section{A regularized semi-implicit finite difference method}
\label{sec:scheme}
In this section, we study the approximation properties of a finite difference method
for solving the regularized model \eqref{RLSE}. For simplicity of
notation, we set $\lambda=1$ and only present the numerical method for
the  RLogSE \eqref{RLSE} in 1D, as  extensions to higher dimensions are
straightforward. When $d=1$, we truncate the RLogSE on a bounded
computational interval $\Og=(a,b)$ with homogeneous Dirichlet boundary
condition (here $|a|$ and $b$ are chosen large enough such that the
truncation error is negligible):
\be\label{RLSE1d}
\left\{
\begin{aligned}
&i\p_t u^\ep(x,t)+\p_{xx} u^\ep(x,t)=u^\ep(x,t)\,\ln(\ep+|u^\ep(x,t)|)^2,
\quad x \in \Omega, \quad t>0,\\
&u^\ep(x,0)=u_0(x),\quad x\in\overline\Omega;\qquad u^\ep(a,t)=u^\ep(b,t)=0, \quad t\ge0,
\end{aligned}
\right.
\ee
\subsection{A finite difference scheme and main results on error bounds}
Choose a mesh size $h:=\Dt x=(b-a)/M$ with $M$ being a positive integer and
a time step $\tau:=\Dt t>0$ and denote the grid points and time steps as
$$x_j:=a+jh,\quad j=0,1,\cdots,M;\quad t_k:=k\tau,\quad k=0,1,2,\dots$$
Define the index sets
$$\mathcal {T}_M=\{j \ | \ j=1,2,\cdots,M-1\},\quad
\mathcal{T}_M^0=\{j\ |\ j=0,1,\cdots,M\}.$$
Let $u^{\ep,k}_j$  be the approximation of $u^{\ep}(x_j,t_k)$, and
denote $u^{\ep,k}=(u^{\ep,k}_0, u^{\ep,k}_1, \ldots, u^{\ep,k}_M)^T\in
\mathbb{C}^{M+1}$ as
the numerical solution vector at $t=t_k$. Define the standard finite difference operators
\[
\dt_t^c u_j^k=\fl{u_j^{k+1}-u_j^{k-1}}{2\tau},\quad
\dt_x^+u_j^k=\fl{u_{j+1}^k-u_j^k}{h},\quad
\dt_x^2u_j^k=\fl{u_{j+1}^k-2u_j^k+u_{j-1}^k}{h^2}.
\]
Denote
$$X_M=\left\{v=\left(v_0,v_1,\ldots,v_M\right)^T\  | \ v_0=v_M=0\right\} \subseteq \mathbb{C}^{M+1},$$
equipped with inner products and norms defined as (recall that
$u_0=v_0=u_M=v_M=0$ by Dirichlet boundary condition)
\be\label{innorm}
\begin{split}
&(u, v)=h\sum\limits_{j=1}^{M-1}u_j \overline{v_j}, \quad \<u,v\>=h\sum\limits_{j=0}^{M-1}u_j \overline{v_j}, \quad
\|u\|_\infty=\sup\limits_{j\in \mathcal{T}_M^0}|u_j|;\\
&\|u\|^2=(u, u),\quad |u|_{H^1}^2=\<\dt_x^+u, \dt_x^+u\>,
  \quad \|u\|_{H^1}^2=\|u\|^2+|u|_{H^1}^2.
\end{split}
\ee
Then we have for $u$, $v\in X_M$,
\be\label{innpX_M}
(-\dt_x^2 u,v)=\<\dt_x^+ u,\dt_x^+ v\>=(u,-\dt_x^2 v).
\ee

Consider a semi-implicit finite difference (SIFD) discretization of \eqref{RLSE1d} as following
\be\label{scheme}
i\dt_t^c u_j^{\ep,k}=-\fl{1}{2}\dt_x^2(u_j^{\ep,k+1}+u_j^{\ep,k-1})+u_j^{\ep,k}\ln(\ep+|u_j^{\ep,k}|)^2,
\quad j\in \mathcal{T}_M,\quad k\ge 1.
\ee
The boundary and initial conditions are discretized as
\be\label{ib}
u_0^{\ep,k}=u_M^{\ep,k}=0,\quad k\ge 0;\quad u_j^{\ep,0}=u_0(x_j),\quad j\in\mathcal{T}_M^0.
\ee
In addition, the first step $u_j^{\ep,1}$ can be obtained via the Taylor expansion as
\be\label{1stp}
u_j^{\ep,1}=u_j^{\ep,0}+\tau u_1(x_j),\quad j\in\mathcal{T}_M^0,
\ee
where
$$u_1(x):=\p_t u^\ep(x,0)=i\left[u_0''(x)-u_0(x)\ln(\ep+|u_0(x)|)^2\right],
\quad a\le x\le b.$$

Let $0<T<T_{\rm max}$ with $T_{\rm max}$ the maximum existence time of the solution $u^\ep$ to the problem \eqref{RLSE1d} for a fixed $0\le \ep\ll1$. By using the standard
von Neumann analysis, we can show that the discretization
\eqref{scheme} is conditionally stable under the stability condition
\be\label{stab}
0<\tau\le \fl{1}{2\max\{|\ln \ep|, \ln(\ep+\max\limits_{j\in\mathcal{T}_M}|u_j^{\ep,k}|)\}},\qquad
0\le k\le \frac{T}{\tau}.
\ee

Define the error functions $e^{\ep,k}\in X_M$ as
\be
e^{\ep,k}_j=u^\ep(x_j,t_k)-u_j^{\ep,k},\quad j\in \mathcal{T}_M^0,\quad 0\le k\le \frac{T}{\tau},
\ee
where $u^\ep$ is the solution of \eqref{RLSE1d}.
Then we have the following error estimates for \eqref{scheme} with
\eqref{ib} and \eqref{1stp}.

\begin{theorem}[\textbf{Main result}]\label{thm1}
Assume that the solution $u^\ep$ is smooth enough over $\Omega_T:=\Omega\times [0,T]$, i.e.
\[(A)\hskip13mm  u^\ep\in C\left([0,T]; H^5(\Omega)\right)\cap C^2\left([0,T]; H^4(\Omega)\right)\cap C^3\left([0,T]; H^2(\Omega)\right),
\hskip13mm \]
and there exist $\ep_0>0$ and $C_0>0$ independent of $\ep$ such that
\[\|u^\ep\|_{L^\infty(0,T; H^5(\Omega))} +
\|\partial_t^2u^\ep\|_{L^\infty(0,T; H^4(\Omega))}+
\|\partial_t^3u^\ep\|_{L^\infty(0,T; H^2(\Omega))}  \le C_0,
\]
uniformly in $0\le \ep\le \ep_0$. Then there exist $h_0>0$ and
$\tau_0>0$ sufficiently small with $h_0\sim \sqrt{\ep} e^{-CT|\ln (\ep)|^2}$ and $\tau_0\sim \sqrt{\ep} e^{-CT|\ln (\ep)|^2}$ such that, when $0<h\le h_0$ and
$0<\tau\le \tau_0$ satisfying the stability condition \eqref{stab}, we
have the following error estimates
\be
\label{esti2}
\begin{split}
&\|e^{\ep,k}\|\le C_1(\ep,T)(h^2+\tau^2),\qquad 0\le k\le \frac{T}{\tau},\\
&\|e^{\ep,k}\|_{H^1}\le C_2(\ep,T)(h^2+\tau^2), \quad \|u^{\ep,k}\|_\infty\le \Lambda+1,
\end{split}
\ee
where $\Lambda=\|u^\ep\|_{L^\infty( \Omega_T)}$,
$C_1(\ep,T)\sim e^{CT|\ln(\ep)|^2}$, $C_2(\ep,T)\sim
\frac{1}{\ep}e^{CT|\ln(\ep)|^2}$ and $C$ depends on $C_0$.
\end{theorem}

The error bounds in this Theorem show not only the
quadratical convergence  in terms of the
mesh size $h$ and time step $\tau$ but also how the
explicit dependence on the regularization parameter
$\ep$. Here we remark that the  Assumption (A) is
valid at least in the case of taking Gaussian as the initial datum.


Define the error functions $\widetilde e^{\ep,k}\in X_M$ as
\be
\widetilde e^{\ep,k}_j=u(x_j,t_k)-u_j^{\ep,k},\quad j\in \mathcal{T}_M^0,\quad 0\le k\le \frac{T}{\tau},
\ee
where $u^\ep$ is the solution of the LogSE \eqref{LSE} with $\Omega=(a,b)$.
Combining Proposition \ref{prop1} and Theorem \ref{thm1},
we immediately obtain (see an illustration in the following diagram):
\[
\xymatrixcolsep{8pc}\xymatrix{
u^{\ep,k}\ar[r]^{\quad O(h^2+\tau^2)} \ar@{-->}[rd]^{
}_{\hspace{-9mm}O(\ep)+ O(h^2+\tau^2)}&
u^\ep(\cdot, t_k)\ar[d]^{O(\ep)} \\
&u(\cdot, t_k)}
\]

\begin{corollary}\label{cor1}
Under the assumptions of Proposition \ref{prop1} and Theorem \ref{thm1},
we have the following error estimates
\be\label{eror}
\begin{split}
&\|\widetilde{e}^{\ep,k}\|\le C_1\ep+C_1(\ep, T)(h^2+\tau^2),\quad\\
&\|\widetilde{e}^{\ep,k}\|_{H^1}\le C_2\ep^{1/2}+C_2(\ep, T)(h^2+\tau^2),\qquad 0\le k\le \frac{T}{\tau},
\end{split}
\ee
where $C_1$ and $C_2$ are presented as in Proposition \ref{prop1}, and $C_1(\ep, T)$ and $C_2(\ep, T)$ are given in Theorem \ref{thm1}.
\end{corollary}

\subsection{Error estimates}
Define the local truncation error $\xi_j^{\ep,k}\in X_M$ for $k\ge1$ as
\be\label{local}
\begin{split}
\xi_j^{\ep,k}&=i\dt_t^c u^\ep(x_j,t_k)+\fl{1}{2}\left(\dt_x^2u^\ep(x_j,t_{k+1})+
\dt_x^2u^\ep(x_j,t_{k-1})\right)\\
&\quad-u^\ep(x_j,t_k)\ln(\ep+|u^\ep(x_j,t_k)|)^2, \qquad j\in \mathcal{T}_M, \quad 1\le k<\frac{T}{\tau},
\end{split}
\ee
then we have the following bounds for the local truncation error.

\begin{lemma}[Local truncation error]\label{local_e}
Under Assumption (A), we have
\[\|\xi^{\ep,k}\|_{H^1}\lesssim h^2+\tau^2,\quad 1\le k<\fl T \tau.\]
\end{lemma}
\begin{proof}
By Taylor expansion, we have
\be\label{tmp1}
\xi_j^{\ep,k}=\fl{i\tau^2}{4}\alpha_j^{\ep,k}+\fl{\tau^2}{2}\beta_j^{\ep,k}+\fl{h^2}{12}\gamma_j^{\ep,k},
\ee
where
\begin{align*}
\alpha_j^{\ep,k}&=\int_{-1}^1(1-|s|)^2\p_t^3u^\ep(x_j,t_k+s \tau)ds,\quad
\beta_j^{\ep,k}=\int_{-1}^1(1-|s|)\p_t^2u_{xx}^\ep(x_j,t_k+s \tau)ds,\\
\gamma_j^{\ep,k}&=\int_{-1}^1(1-|s|)^3
\left(\p_x^4u^\ep(x_j+sh,t_{k+1})+\p_x^4u^\ep(x_j+sh,t_{k-1})\right)ds.
\end{align*}
By the Cauchy-Schwarz inequality, we can get that
\begin{align*}
\|\alpha^{\ep,k}\|^2&=h\sum\limits_{j=1}^{M-1}|\alpha_j^{\ep,k}|^2
\le h\int_{-1}^1 (1-|s|)^4ds\sum\limits_{j=1}^{M-1}
\int_{-1}^1\left|\p_t^3u^\ep(x_j,t_k+s\tau)\right|^2ds\\
&=\fl{2}{5}\Big[\int_{-1}^1\|\p_t^3u^\ep(\cdot,t_k+s\tau)\|_{L^2(\Omega)}^2ds\Big.\\
&\qquad\Big.-\int_{-1}^1\sum\limits_{j=0}^{M-1}\int_{x_j}^{x_{j+1}}(|\p_t^3u^\ep(x,t_k+s\tau)|^2
-|\p_t^3 u^\ep(x_j,t_k+s\tau)|^2)dxds\Big]\\
&=\fl{2}{5}\Big[\int_{-1}^1\|\p_t^3u^\ep(\cdot,t_k+s\tau)\|_{L^2(\Omega)}^2ds\Big.\\
&\qquad\Big.-\int_{-1}^1\sum\limits_{j=0}^{M-1}\int_{x_j}^{x_{j+1}}
\int_{x_j}^\og\p_x|\p_t^3u^\ep(x',t_k+s\tau)|^2
dx'd\og ds\Big]\\
&\le\fl{2}{5}\int_{-1}^1\Big[\|\p_t^3u^\ep(\cdot,t_k+s\tau)\|_{L^2(\Omega)}^2\Big.\\
&\qquad\qquad+\Big.2h \|\p_t^3u_x^\ep(\cdot,t_k+s\tau)\|_{L^2(\Omega)}\|\p_t^3u^\ep(\cdot,t_k+s\tau)\|_{L^2(\Omega)}\Big]ds\\
&\le\max\limits_{0\le t\le T}\left(\|\p_t^3 u^\ep\|_{L^2(\Omega)}+h\|\p_t^3 u_x^\ep\|_{L^2(\Omega)}\right)^2,
\end{align*}
which yields that when $h\le 1$,
\[\|\alpha^{\ep,k}\|\le
\|\p_t^3u^\ep\|_{L^\infty(0,T;H^1(\Omega))}.\]
 Applying the similar approach, it can be established that
\[\|\beta^{\ep,k}\|\le 2\|\p_t^2u^\ep\|_{L^\infty(0,T;H^3(\Omega))}.\]
On the other hand, we can obtain that
\begin{align*}
\|\gamma^{\ep,k}\|^2&
\le h\int_{-1}^1 (1-|s|)^6ds\sum\limits_{j=1}^{M-1}
\int_{-1}^1\left|\p_x^4u^\ep(x_j+sh,t_{k+1})+\p_x^4u^\ep(x_j+sh,t_{k-1})\right|^2ds\\&\le \fl{4h}{7}\sum\limits_{j=1}^{M-1}
\int_{-1}^1\left(\left|\p_x^4u^\ep(x_j+sh,t_{k+1})\right|^2+
\left|\p_x^4u^\ep(x_j+sh,t_{k-1})\right|^2\right)ds\\
&\le \fl{8}{7}\left(\|\p_x^4u^\ep(\cdot,t_{k-1})\|^2_{L^2(\Omega)}+
\|\p_x^4u^\ep(\cdot,t_{k+1})\|^2_{L^2(\Omega)}\right)\\
&\le 4\|u^\ep\|^2_{L^\infty(0,T;H^4(\Omega))},
\end{align*}
which implies that $\|\gamma^{\ep,k}\|\le 2\|u^\ep\|_{L^\infty(0,T;H^4(\Omega))}$.
Hence by Assumption (A), we get
\begin{align*}
\|\xi^{\ep,k}\|&\lesssim \tau^2\left(\|\p_t^3u^\ep\|_{L^\infty(0,T;H^1(\Omega))}+
\|\p_t^2u^\ep\|_{L^\infty(0,T;H^3(\Omega))}\right)+h^2\|u^\ep\|_{L^\infty(0,T;H^4(\Omega))}\\
&\lesssim_{C_0} \tau^2+h^2.
\end{align*}
Applying $\dt_x^+$ to $\xi^{\ep,k}$ and using the same approach, we can get that
\begin{align*}
|\xi^{\ep,k}|_{H^1}&\lesssim \tau^2\left(\|\p_t^3u^\ep\|_{L^\infty(0,T;H^2(\Omega))}+
\|\p_t^2u^\ep\|_{L^\infty(0,T;H^4(\Omega))}\right)+h^2\|u^\ep\|_{L^\infty(0,T;H^5(\Omega))}\\
&\lesssim_{C_0} \tau^2+h^2,
\end{align*}
which completes the proof.
\end{proof}

For the first step, we have the following estimates.
\begin{lemma}[Error bounds for $k=1$] \label{initial-e}
Under Assumption (A), the first step errors of the discretization \eqref{1stp} satisfy
\[
e^{\ep,0}=0,\quad
\|e^{\ep,1}\|_{H^1}\lesssim\tau^2.
\]
\end{lemma}
\begin{proof}
By the definition of $u^{\ep,1}_j$ in \eqref{1stp}, we have
\begin{align*}
e_j^{\ep,1}=\tau^2\int_0^1(1-s)u_{tt}^\ep(x_j,s\tau)ds,
  \end{align*}
which implies that
\[
\|e^{\ep,1}\|\lesssim \tau^2\|\p_t^2u^\ep\|_{L^\infty(0,T; H^1(\Omega))}\lesssim \tau^2,\quad
|e^{\ep,1}|_{H^1}\lesssim \tau^2\|\p_t^2u^\ep\|_{L^\infty(0,T; H^2(\Omega))}\lesssim\tau^2,
\]
and the proof is completed.
\end{proof}
\bigskip

\begin{proof}[Proof of Theorem \ref{thm1}] We prove \eqref{esti2} by
  induction. It follows from Lemma \ref{initial-e} that \eqref{esti2}
  is true for $k=0, 1$.

Assume \eqref{esti2} is valid for $k\le n\le \fl{T}{\tau}-1$. Next we need to show that \eqref{esti2} still holds for $k=n+1$.
Subtracting \eqref{scheme} from \eqref{local}, we get the error equations
\be\label{eeq}
i\dt_t^c e^{\ep,m}_j=-\fl{1}{2}(\dt_x^2e_j^{\ep,m+1}+\dt_x^2e_j^{\ep,m-1})+r_j^{\ep,m}+\xi_j^{\ep,m},\quad
j \in \mathcal{T}_M, \quad 1\le m\le \frac{T}{\tau}-1,
\ee
where $r^{\ep,m}\in X_M$ represents the difference between the logarithmic nonlinearity
\be\label{rk}
r_j^{\ep,m}=u^\ep(x_j,t_m)\ln(\ep+|u^\ep(x_j,t_m)|)^2-u_j^{\ep,m}\ln(\ep+|u_j^{\ep,m}|)^2
, \quad 1\le m\le \frac{T}{\tau}-1.
\ee
Multiplying both sides of \eqref{eeq} by $2\tau\,(\overline{e_{j}^{\ep,m+1}+e_{j}^{\ep,m-1}})$, summing together for $j\in \mathcal{T}_M$ and taking the imaginary parts, we obtain for $1\le m< T/\tau$,
\be\label{eq1}
\begin{aligned}
\|e^{\ep,m+1}\|^2-\|e^{\ep,m-1}\|^2&=2\tau \, \mathrm{Im}(r^{\ep,m}+\xi^{\ep,m},e^{\ep,m+1}+e^{\ep,m-1})\\
&\le 2\tau\left(\|r^{\ep,m}\|^2+\|\xi^{\ep,m}\|^2+\|e^{\ep,m+1}\|^2+\|e^{\ep,m-1}\|^2\right).
\end{aligned}
\ee
Summing \eqref{eq1} for $m=1, 2, \cdots, n$ ($n\le \fl{T}{\tau}-1$), we obtain
\begin{align}
\|e^{\ep,n+1}\|^2+\|e^{\ep,n}\|^2&\le \|e^{\ep,0}\|^2+\|e^{\ep,1}\|^2+2\tau\|e^{\ep,n+1}\|^2+2\tau\sum\limits_{m=0}^{n-1}
(\|e^{\ep,m}\|^2+\|e^{\ep,m+1}\|^2)\nn\\
&\quad+2\tau\sum\limits_{m=1}^n\left(\|r^{\ep,m}\|^2+\|\xi^{\ep,m}\|^2\right).\label{sum1}
\end{align}
For $m\le n$, when $|u_j^{\ep,m}|\le |u^\ep(x_j,t_m)|$, we write $r_j^{\ep,m}$ as
\begin{align*}
|r_j^{\ep,m}|&=\Big|e_j^{\ep,m}\ln(\ep+|u^\ep(x_j,t_m)|)^2+2
u_j^{\ep,m}\ln\Big(\fl{\ep+|u^\ep(x_j,t_m)|}{\ep+|u_j^{\ep,m}|}\Big)\Big|\\
&\le 2\max\{\ln(\ep^{-1}), |\ln(\ep+\Lambda)|\}|e_j^{\ep,m}|+2|u_j^{\ep,m}|\ln\Big(1+\fl{|u^\ep(x_j,t_m)|-|u_j^{\ep,m}|}
{\ep+|u_j^{\ep,m}|}\Big)\\
&\le 2|e_j^{\ep,m}|(1+\max\{\ln(\ep^{-1}), |\ln(\ep+\Lambda)|\}).
\end{align*}
On the other hand, when $|u^\ep(x_j,t_m)|\le |u_j^{\ep,m}|$, we write $r_j^{\ep,m}$ as
\begin{align*}
|r_j^{\ep,m}|&=\Big|e_j^{\ep,m}\ln(\ep+|u_j^{\ep,m}|)^2+2
u^\ep(x_j,t_m)\ln\Big(\fl{\ep+|u^\ep(x_j,t_m)|}{\ep+|u_j^{\ep,m}|}\Big)\Big|\\
&\le  2\max\{\ln(\ep^{-1}), |\ln(\ep+1+\Lambda)|\}|e_j^{\ep,m}|\\
&\quad+2|u^\ep(x_j,t_m)|\ln\Big(1+\fl{|u_j^{\ep,m}|-|u^\ep(x_j,t_m)|}
{\ep+|u^\ep(x_j,t_m)|}\Big)\\
&\le 2|e_j^{\ep,m}|(1+\max\{\ln(\ep^{-1}), |\ln(\ep+1+\Lambda)|\}),
\end{align*}
where we use the assumption that $\|u^{\ep,m}\|_\infty\le \Lambda+1$ for $m\le n$.
Thus it follows that
\[\|r^{\ep,m}\|^2\lesssim |\ln(\ep)|^2\|e^{\ep,m}\|^2,\]
when $\ep$ is sufficiently small.
Thus when $\tau\le\fl{1}{2}$, by using Lemmas \ref{local_e}, \ref{initial-e} and \eqref{sum1}, we have
\begin{align*}
\|e^{\ep,n+1}\|^2+\|e^{\ep,n}\|^2&\lesssim \|e^{\ep,0}\|^2+\|e^{\ep,1}\|^2+\tau\sum\limits_{m=0}^{n-1}
(\|e^{\ep,m}\|^2+\|e^{\ep,m+1}\|^2)\\
&\quad+\tau\sum\limits_{m=1}^n\left(\|r^{\ep,m}\|^2+\|\xi^{\ep,m}\|^2\right)\\
&\lesssim (h^2+\tau^2)^2+\tau|\ln(\ep)|^2\sum\limits_{m=0}^{n-1}(\|e^{\ep,m}\|^2+\|e^{\ep,m+1}\|^2).
\end{align*}
We emphasize here that the implicit multiplicative constant in this
inequality depends only on $C_0$, but not on $n$.
Applying the discrete Gronwall inequality, we can conclude that
\[\|e^{\ep,n+1}\|^2\lesssim e^{CT|\ln(\ep)|^2}(h^2+\tau^2)^2,\]
for some $C$ depending on $C_0$,
which gives the error bound for $\|e^{\ep,k}\|$ with $k=n+1$ in
\eqref{esti2} immediately.

To estimate $|e^{\ep,n+1}|_{H^1}$, multiplying both sides of
\eqref{eeq} by $2(\overline{e_j^{\ep,m+1}-e_j^{\ep,m-1}})$ for $m\le
n$, summing together for $j\in \mathcal{T}_M$ and taking the real
parts, we obtain
\begin{align}
&|e^{\ep,m+1}|_{H^1}^2-|e^{\ep,m-1}|_{H^1}^2\nn\\
&\quad=-2  \,\mathrm{Re}\big(r^{\ep,m}+\xi^{\ep,m}, e^{\ep,m+1}-e^{\ep,m-1}\big)\nn\\
&\quad=2\tau\,\mathrm{Im}\big(r^{\ep,m}+\xi^{\ep,m}, -\dt_x^2(e^{\ep,m+1}+e^{\ep,m-1})\big)\nn\\
&\quad=2\tau\,\mathrm{Im}\big<\dt_x^+(r^{\ep,m}+\xi^{\ep,m}), \dt_x^+(e^{\ep,m+1}+e^{\ep,m-1})\big>\nn\\
&\quad\le 2\tau\left(|r^{\ep,m}|_{H^1}^2+|\xi^{\ep,m}|_{H^1}^2+|e^{\ep,m+1}|_{H^1}^2
+|e^{\ep,m-1}|_{H^1}^2\right).\label{eq2}
\end{align}
To give the bound for $\dt_x^+r^{\ep,m}$, for simplicity of notation, denote
\[u^{\ep,m}_{j,\tht}=\tht u^\ep(x_{j+1},t_m)+(1-\tht)u^\ep(x_j,t_m),\quad
v^{\ep,m}_{j,\tht}=\tht v_{j+1}^{\ep,m}+(1-\tht)v_j^{\ep,m},\]
for $j\in\mathcal{T}_M$ and $\tht\in [0,1]$.
Then we have
\begin{align*}
\dt_x^+r_j^{\ep,m}&=2\dt_x^+(u^\ep(x_j,t_m)
\ln(\ep+|u^\ep(x_j,t_m)|))-2\dt_x^+(u_j^{\ep,m}\ln(\ep+|u_j^{\ep,m}|))\\
&=\fl{2}{h}\int_0^1[u_{j,\tht}^{\ep,m}\ln(\ep+|u_{j,\tht}^{\ep,m}|)]'(\tht)d\tht-
\fl{2}{h}\int_0^1[v_{j,\tht}^{\ep,m}\ln(\ep+|v_{j,\tht}^{\ep,m}|)]'(\tht)d\tht\\
&=I_1+I_2+I_3,
\end{align*}
where
\begin{align*}
I_1&:=2\dt_x^+u^\ep(x_j,t_m)\int_0^1\ln(\ep+|u_{j,\tht}^{\ep,m}|)d\tht-
2\dt_x^+u_j^{\ep,m}\int_0^1\ln(\ep+|v_{j,\tht}^{\ep,m}|)d\tht,\\
I_2&:=\dt_x^+u^\ep(x_j,t_m)\int_0^1 \fl{|u_{j,\tht}^{\ep,m}|}{\ep+|u_{j,\tht}^{\ep,m}|}d\tht-
\dt_x^+u_j^{\ep,m}\int_0^1 \fl{|v_{j,\tht}^{\ep,m}|}
{\ep+|v_{j,\tht}^{\ep,m}|}d\tht,\\
I_3&:=\dt_x^+\overline{u^\ep(x_j,t_m)}\int_0^1\fl{(u_{j,\tht}^{\ep,m})^2}
{|u_{j,\tht}^{\ep,m}|(\ep+|u_{j,\tht}^{\ep,m}|)}d\tht-
\dt_x^+\overline{u_j^{\ep,m}}\int_0^1\fl{(v_{j,\tht}^{\ep,m})^2}
{|v_{j,\tht}^{\ep,m}|(\ep+|v_{j,\tht}^{\ep,m}|)}d\tht.
\end{align*}
Then we estimate $I_1$, $I_2$ and $I_3$, separately.
Similar as before, we have
\begin{align*}
\left|I_1\right|&\le
                  2|\dt_x^+u^\ep(x_j,t_m)|\;\int_0^1\Big|\ln
\Big(\fl{\ep+|u_{j,\tht}^{\ep,m}|}{\ep+|v_{j,\tht}^{\ep,m}|}\Big)
\Big|d\tht+2\left|\dt_x^+e_j^{\ep,m}\right|\;\int_0^1
\Big|\ln(\ep+|v_{j,\tht}^{\ep,m}|)\Big|d\tht\\
&=2|\dt_x^+u^\ep(x_j,t_m)|\;\int_0^1\ln\Big(1+\fl{\left\lvert|
u_{j,\tht}^{\ep,m}|-|v_{j,\tht}^{\ep,m}|\right\rvert}{\ep+\min\{
|u_{j,\tht}^{\ep,m}|, |v_{j,\tht}^{\ep,m}|\}}\Big)d\tht\\
&\quad+2\left|\dt_x^+e_j^{\ep,m}\right|\;\int_0^1\big|\ln(\ep+|v_{j,\tht}^{\ep,m}|)\big|d\tht\\
&\le\fl{2}{\ep}|\dt_x^+u^\ep(x_j,t_m)|\left(|e_j^{\ep,m}|+|e_{j+1}^{\ep,m}|\right)+
2\left|\dt_x^+e_j^{\ep,m}\right|\max\{\ln(\ep^{-1}), |\ln(\ep+1+\Lambda)|\}\\
&\lesssim \fl{1}{\ep}\left(|e_j^{\ep,m}|+|e_{j+1}^{\ep,m}|\right)+\ln(\ep^{-1})
\left|\dt_x^+e_j^{\ep,m}\right|,
\end{align*}
and
\begin{align*}
\left|I_2\right|&= \Big|\dt_x^+u^\ep(x_j,t_m)\int_0^1\Big(\fl{|u_{j,\tht}^{\ep,m}|}{\ep+|u_{j,\tht}^{\ep,m}|}-
\fl{|v_{j,\tht}^{\ep,m}|}{\ep+|v_{j,\tht}^{\ep,m}|}\Big)
d\tht+\dt_x^+e_j^{\ep,m}\int_0^1\fl{|v_{j,\tht}^{\ep,m}|}{\ep+|v_j^{\ep,m}|}d\tht\Big|\\
&\le |\dt_x^+e_j^{\ep,m}|+|\dt_x^+u^\ep(x_j,t_m)|
\int_0^1\fl{\ep|u_{j,\tht}^{\ep,m}-v_{j,\tht}^{\ep,m}|}{(\ep+|u_{j,\tht}^{\ep,m}|)(\ep+|v_{j,\tht}^{\ep,m}|)}d\tht\\
&\le |\dt_x^+e_j^{\ep,m}|+\fl{|\dt_x^+u^\ep(x_j,t_m)|}{\ep}\int_0^1|u_{j,\tht}^{\ep,m}-v_{j,\tht}^{\ep,m}|d\tht\\
&\lesssim |\dt_x^+e_j^{\ep,m}|+\fl{1}{\ep}\left(|e_j^{\ep,m}|+|e_{j+1}^{\ep,m}|\right).
\end{align*}
In view of the inequality that
\begin{align*}
&\Big|\fl{(u_{j,\tht}^{\ep,m})^2}
{|u_{j,\tht}^{\ep,m}|(\ep+|u_{j,\tht}^{\ep,m}|)}-\fl{(v_{j,\tht}^{\ep,m})^2}
{|v_{j,\tht}^{\ep,m}|(\ep+|v_{j,\tht}^{\ep,m}|)}\Big|\\
&=\Big|
\fl{(u_{j,\tht}^{\ep,m})^2-u_{j,\tht}^{\ep,m} v_{j,\tht}^{\ep,m}}{|u_{j,\tht}^{\ep,m}|(\ep+|u_{j,\tht}^{\ep,m}|)}+
\fl{u_{j,\tht}^{\ep,m} v_{j,\tht}^{\ep,m}}{|u_{j,\tht}^{\ep,m}|(\ep+|u_{j,\tht}^{\ep,m}|)}-
\fl{(v_{j,\tht}^{\ep,m})^2}{|v_{j,\tht}^{\ep,m}|(\ep+|v_{j,\tht}^{\ep,m}|)}\Big|\\
&\le \fl{|u_{j,\tht}^{\ep,m}-v_{j,\tht}^{\ep,m}|}{\ep}+\fl{\left|u_{j,\tht}^{\ep,m} (v_{j,\tht}^{\ep,m})^2
(\overline{u_{j,\tht}^{\ep,m}}-\overline{v_{j,\tht}^{\ep,m}})+\ep v_{j,\tht}^{\ep,m}(u_{j,\tht}^{\ep,m} |v_{j,\tht}^{\ep,m}|-|u_{j,\tht}^{\ep,m}|v_{j,\tht}^{\ep,m})\right|}{|u_{j,\tht}^{\ep,m}||v_{j,\tht}^{\ep,m}|
(\ep+|u_{j,\tht}^{\ep,m}|)(\ep+|v_{j,\tht}^{\ep,m}|)}\\
&\le \fl{4|u_{j,\tht}^{\ep,m}-v_{j,\tht}^{\ep,m}|}{\ep},
\end{align*}
we can obtain that
\[I_3\lesssim |\dt_x^+e_j^{\ep,m}|+\fl{1}{\ep}\left(|e_j^{\ep,m}|+|e_{j+1}^{\ep,m}|\right).\]
Thus we can conclude that
\[|\dt_x^+ r_j^{\ep,m}|\lesssim \fl{1}{\ep}\left(|e_j^{\ep,m}|+|e_{j+1}^{\ep,m}|\right)+\ln(\ep^{-1})
\left|\dt_x^+e_j^{\ep,m}\right|.\]
Summing \eqref{eq2} for $m=1, 2, \cdots, n$ ($n\le \fl{T}{\tau}-1$), we obtain
\begin{align*}
|e^{\ep,n+1}|_{H^1}^2+|e^{\ep,n}|_{H^1}^2&\le |e^{\ep,0}|_{H^1}^2+|e^{\ep,1}|_{H^1}^2
+\tau\sum\limits_{m=1}^n\left(|r^{\ep,m}|_{H^1}^2+|\xi^{\ep,m}|_{H^1}^2\right)\\
&\quad+\tau |e^{\ep,n+1}|_{H^1}^2+\tau\sum\limits_{m=0}^{n-1}
(|e^{\ep,m}|_{H^1}^2+|e^{\ep,m+1}|_{H^1}^2).
\end{align*}
Thus when $\tau\le 1/2$, by using Lemmas \ref{local_e} and \ref{initial-e}, we have
\begin{align*}
|e^{\ep,n+1}|_{H^1}^2+|e^{\ep,n}|_{H^1}^2&\lesssim |e^{\ep,0}|_{H^1}^2+|e^{\ep,1}|_{H^1}^2+\tau\sum\limits_{m=1}^{n}
\left(\frac{1}{\ep^2}|e^{\ep,m}|_{H^1}^2+|\xi^{\ep,m}|_{H^1}^2\right)\\
&\quad+\tau|\ln(\ep)|^2\sum\limits_{m=0}^{n-1}\left(|e^{\ep,m}|_{H^1}^2+|e^{\ep,m+1}|_{H^1}^2\right)\\
&\hspace{-2mm}\lesssim \fl{e^{CT|\ln(\ep)|^2}}{\ep^2}(h^2+\tau^2)^2+\tau|\ln(\ep)|^2
\sum\limits_{m=0}^{n-1}(|e^{\ep,m}|_{H^1}^2+|e^{\ep,m+1}|_{H^1}^2).
\end{align*}
Applying the discrete Gronwall's inequality, we can get that
\[|e^{\ep,n+1}|_{H^1}^2\lesssim e^{CT|\ln(\ep)|^2}(h^2+\tau^2)^2/\ep^2,\]
which establishes the error estimate for $\|e^{\ep,k}\|_{H^1}$ for $k=n+1$. Finally the boundedness for the solution $u^{\ep,k}$ can be obtained by the triangle inequality
\[\|u^{\ep,k}\|_\infty\le \|u^\ep(\cdot, t_k)\|_{L^\infty(\Omega)}+\|e^{\ep,k}\|_\infty,\]
and the inverse Sobolev inequality \cite{Thomee}
\[\|e^{\ep,k}\|_\infty\lesssim \| e^{\ep,k}\|_{H^1},\]
which completes the proof of Theorem \ref{thm1}.
\end{proof}

\section{Numerical results}
\label{sec:illust}
In this section, we test the convergence rate of the regularized model
\eqref{RLSE} and the SIFD \eqref{scheme}. To this end, we take $d=1$,
$\Omega={\mathbb R}$ and
$\lambda=-1$ in the LogSE \eqref{LSE}
and consider two different initial data:

\smallskip

{Case I}: A Gaussian initial data, i.e. $u_0$ in \eqref{LSE} is chosen as
\be
u_0(x)=\sqrt[4]{-\lambda/\pi} e^{ivx +\frac{\lambda}{2}x^2}, \qquad x\in {\mathbb R},
\ee
with $v=1$. In this case, the LogSE \eqref{LSE}  admits
the moving Gausson solution \eqref{Gausson} with $v=1$ and $b_0=\sqrt[4]{-\lambda/\pi}$ as the exact solution.

\smallskip

{Case II}:  A general initial data, i.e. $u_0$ in \eqref{LSE} is chosen as
\be
\label{soliton}
u_0(x)=\tanh(x) e^{-x^2}, \qquad x\in {\mathbb R},
\ee
which is the multiplication of a dark soliton of the cubic nonlinear Schr\"{o}dinger equation and a Gaussian.  Notice that in this case, the logarithmic  term $\ln |u_0|^2$ is singular at $x=0$.

\smallskip

The RLogSE \eqref{RLSE} is solved numerically by the SIFD \eqref{scheme}  on domains  $\Og=[-12, 12]$ and $\Og=[-16, 16]$ for {Case I} and  {II}, respectively.  To quantify the numerical errors, we introduce  the following error functions:
\be\label{Neror}
\begin{split}
&\widehat{e}^{\ep}(t_k):=u(\cdot,t_k)-u^{\ep}(\cdot,t_k), \qquad
e^{\ep}(t_k):=u^{\ep}(\cdot, t_k)-u^{\ep,k}, \\
&\widetilde{e}^{\ep}(t_k):=u(\cdot, t_k)-u^{\ep,k}, \qquad\quad\,\,\,\,
e_{E}^{\ep}:=|E(u)-E^{\ep}(u^\ep)|.
\end{split}
\ee
Here $u$ and $u^\ep$ are the exact solutions of the LogSE \eqref{LSE} and RLogSE \eqref{RLSE}, respectively, while
$u^{\ep, k}$ is the numerical solution of the RLogSE \eqref{RLSE} obtained by the SIFD \eqref{scheme}. The `exact' solution $u^{\ep}$  is obtained numerically by the SIFD \eqref{scheme} with a very small time step, e.g. $\tau=0.01/2^9$ and a very fine mesh size, e.g. $h=1/2^{15}$.
Similarly, the `exact' solution $u$ in Case II is obtained numerically
by  the SIFD \eqref{scheme} with a very small time step and a very fine mesh size as well as a very small regularization parameter $\ep$, e.g. $\ep=10^{-14}$.
The energy is obtained by the trapezoidal rule for approximating the integrals in the energy  \eqref{conserv} and \eqref{regen}.

\subsection{Convergence rate of the regularized model}
Here we consider the error between the solutions of the RLogSE \eqref{RLSE} and the LogSE \eqref{LSE}. Fig. \ref{fig:Conv_RLSEs_Case1_And_2_Mod1} shows $\|\widehat{e}^{\ep}\|$, $\|\widehat{e}^{\ep}\|_{H^1}$, $\|\widehat{e}^{\ep}\|_\infty$ (the definition of the norms is given in \eqref{innorm}) at time $t=0.5$ for {Cases I} \& {II},
while Fig. \ref{fig:Conv_RLSE_Energy_Error_and_Errors_WRT_t} depicts $e_{E}^{\ep}(0.5)$ for {Cases I} \& {II} and  time evolution of $\widehat{e}^{\ep}(t)$ with different $\ep$ for {Case I}.
For comparison, similar to Fig. \ref{fig:Conv_RLSEs_Case1_And_2_Mod1}, Fig.  \ref{fig:Conv_RLSEs_Case1_And_2_Mod2}
displays the convergent results from \eqref{RLogSE_2} to \eqref{LSE}.

\begin{figure}[htbp!]
\centerline{
\psfig{figure=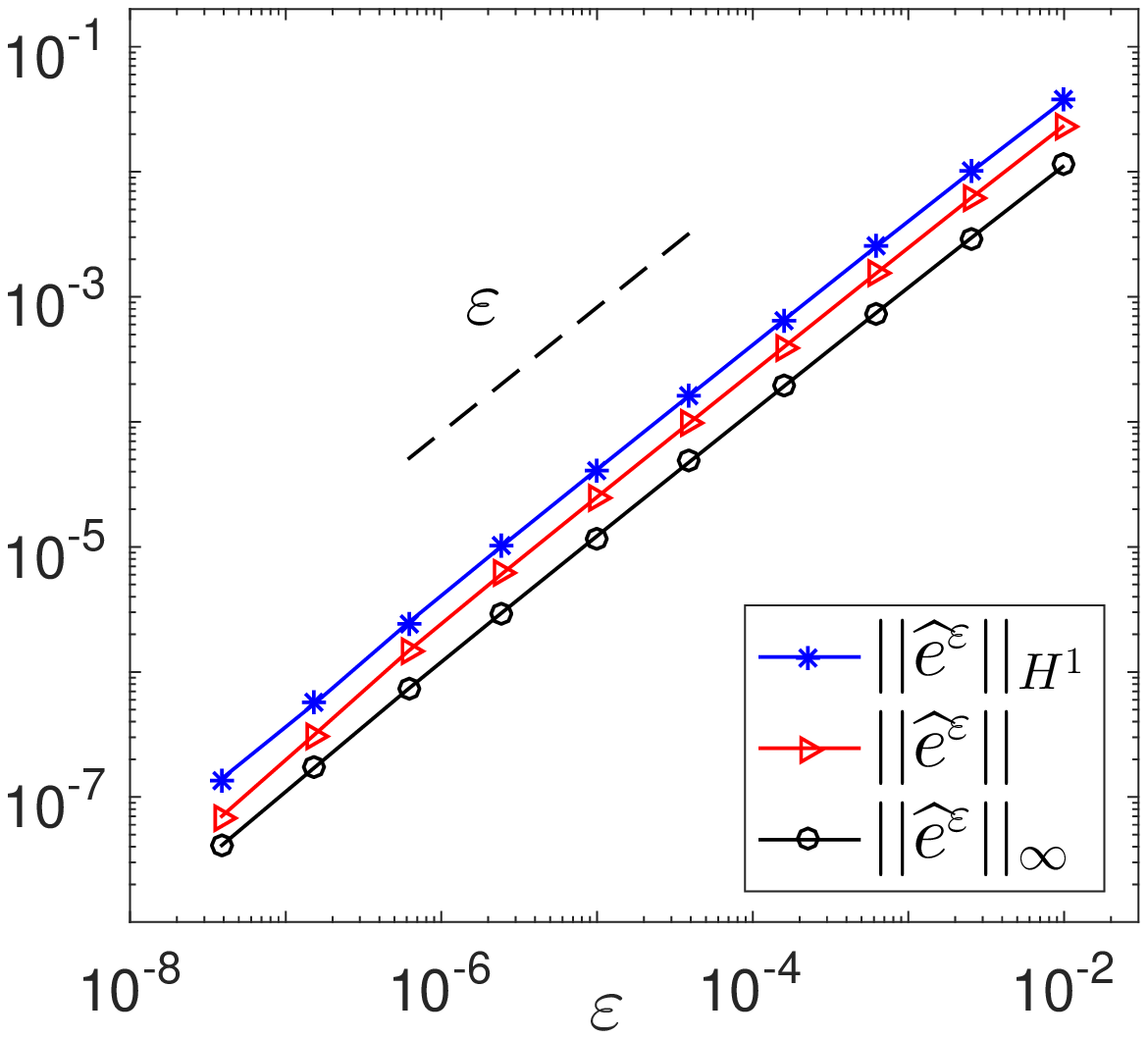,width=2.5in,height=2.5in}
\quad
\psfig{figure=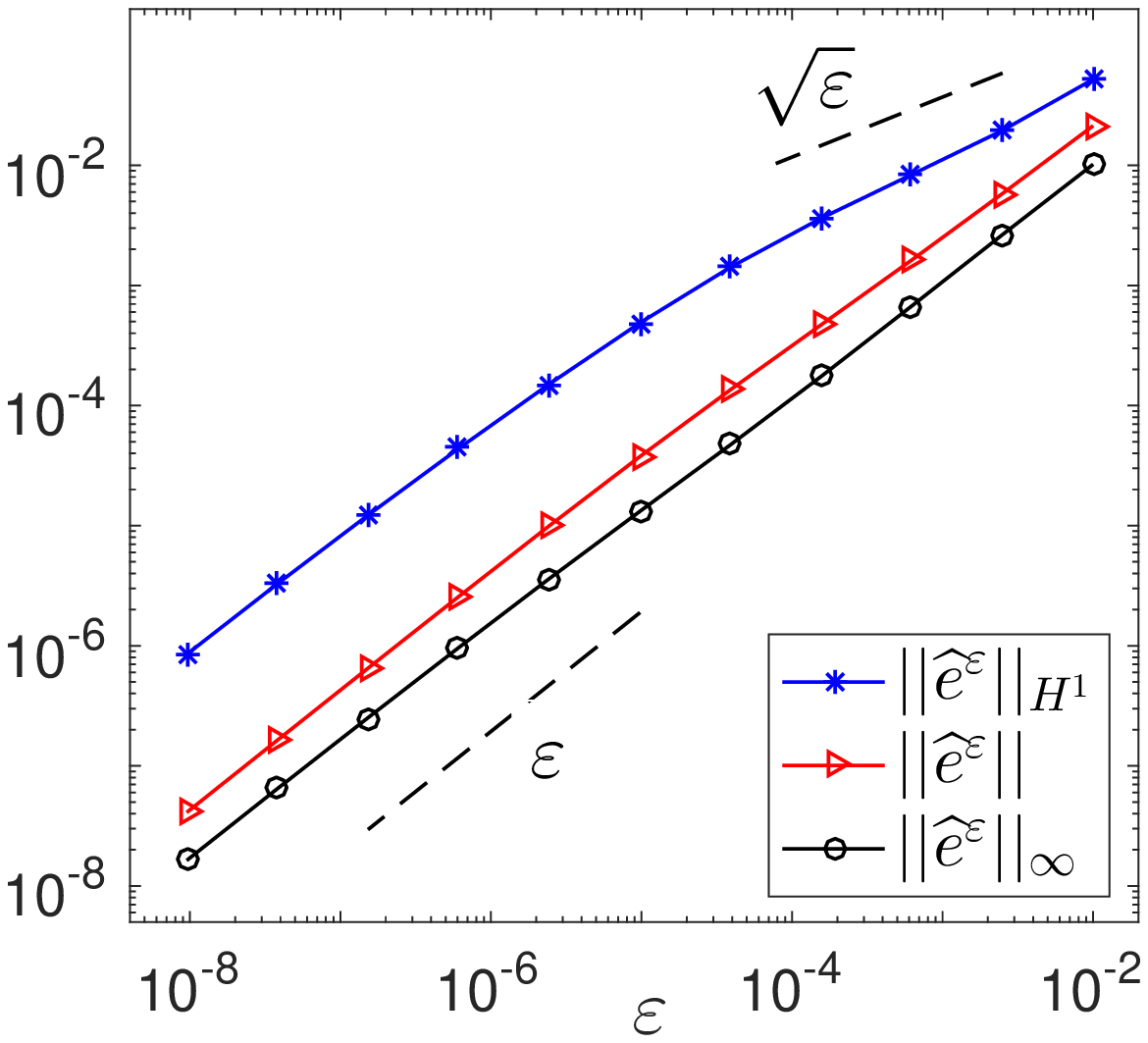,width=2.5in,height=2.5in}}
 \caption{Convergence of the RLogSE \eqref{RLSE} to the LogSE \eqref{LSE},  i.e. the error $\widehat{e}^\ep(0.5)$ in different norms vs the regularization parameter $\ep$  for {Case I} (left) and {Case II} (right).}
\label{fig:Conv_RLSEs_Case1_And_2_Mod1}
\end{figure}

\begin{figure}[htbp!]
\centerline{
\psfig{figure=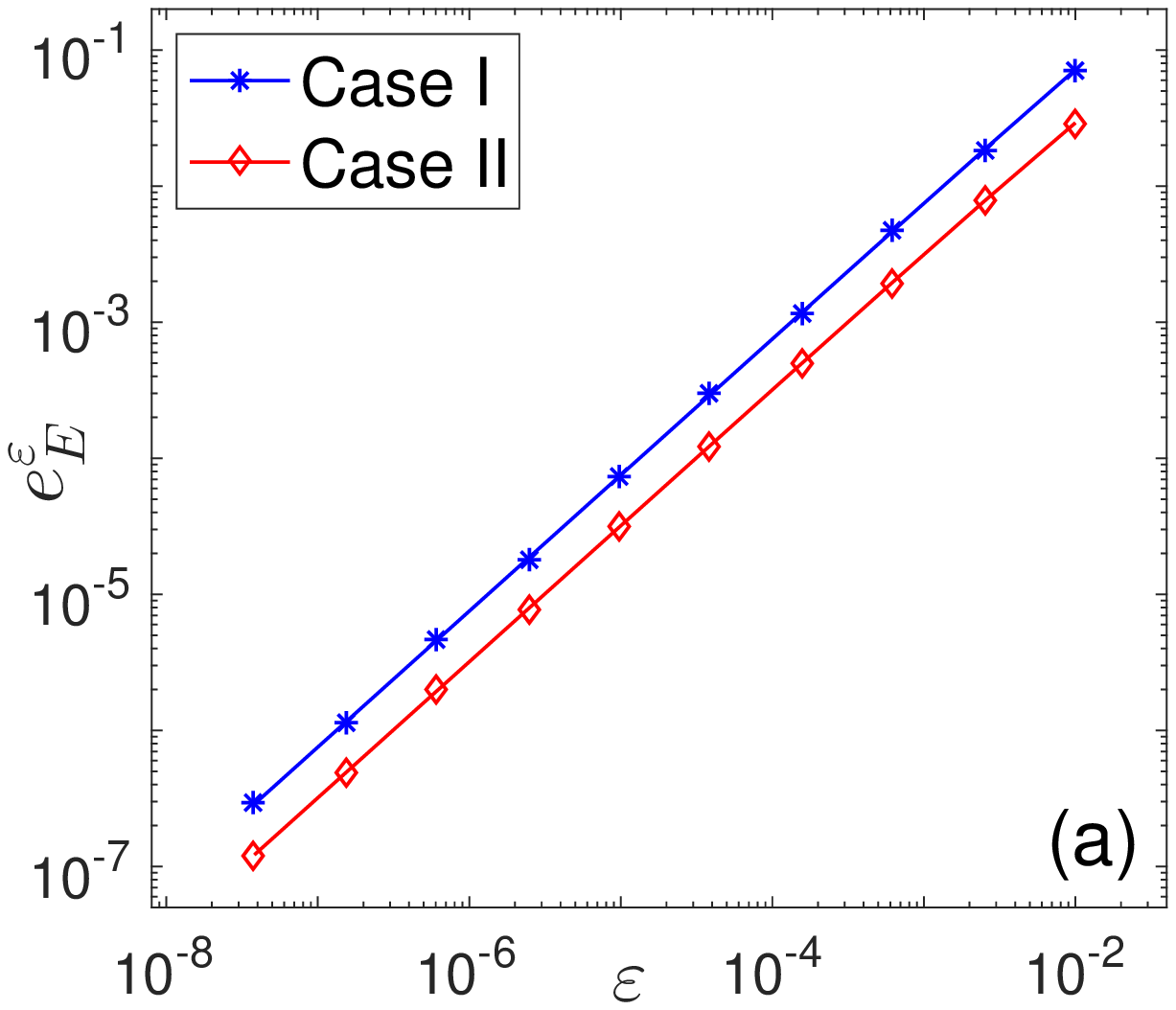,width=2.5in,
height=2.5in}\quad
\psfig{figure=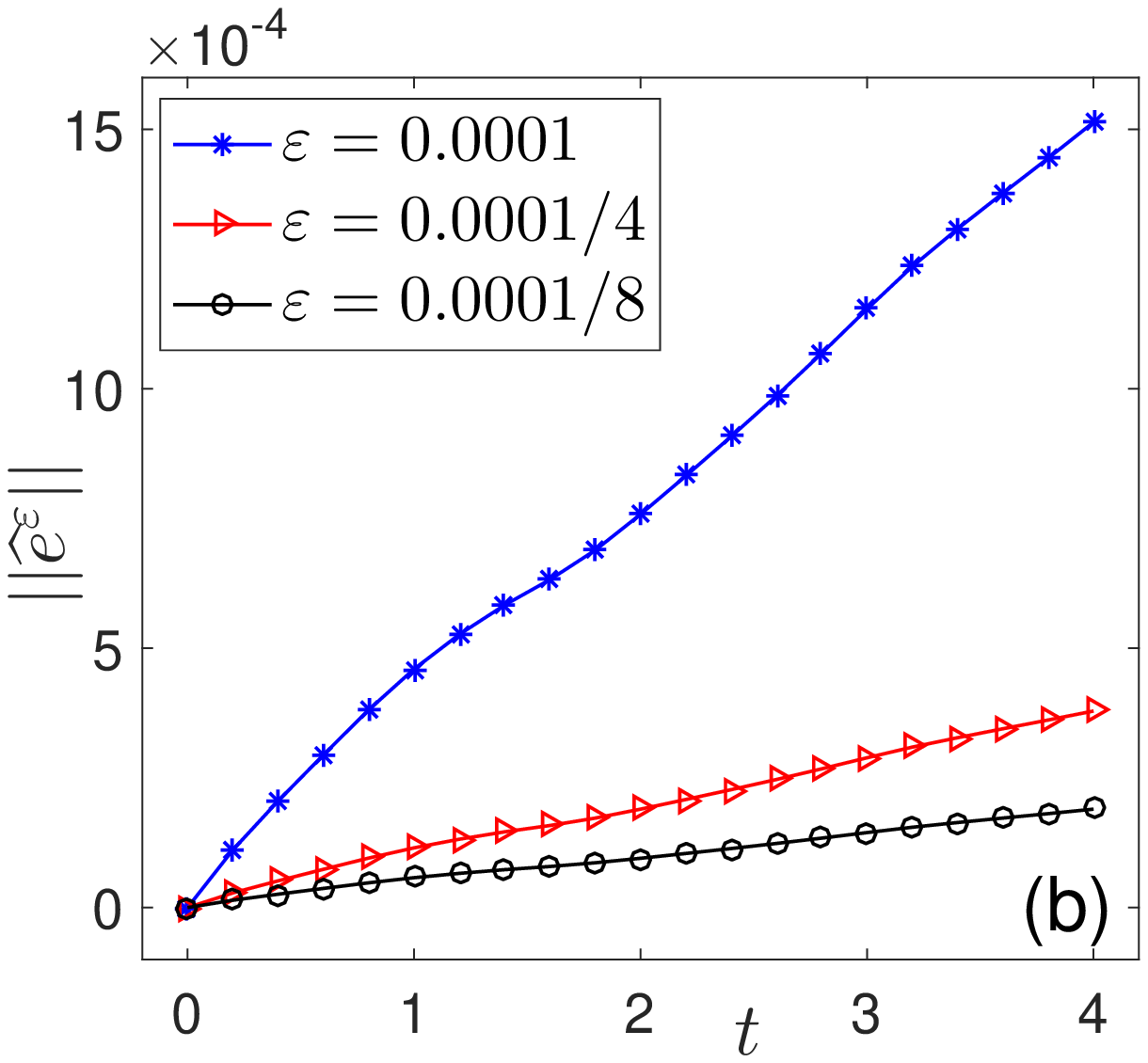,
width=2.5in,height=2.5in}}
 \caption{Convergence of the RLogSE \eqref{RLSE} to the LogSE \eqref{LSE}:
 (a) error in energy $e_{\rm E}^\ep(0.5)$ vs $\ep$ for Cases I \& II,
 and   (b) time evolution of $\|\widehat{e}^\ep(t)\|$ vs time $t$ under different $\ep$  for {Case I}.}
\label{fig:Conv_RLSE_Energy_Error_and_Errors_WRT_t}
\end{figure}

\begin{figure}[htbp!]
\centerline{
\psfig{figure=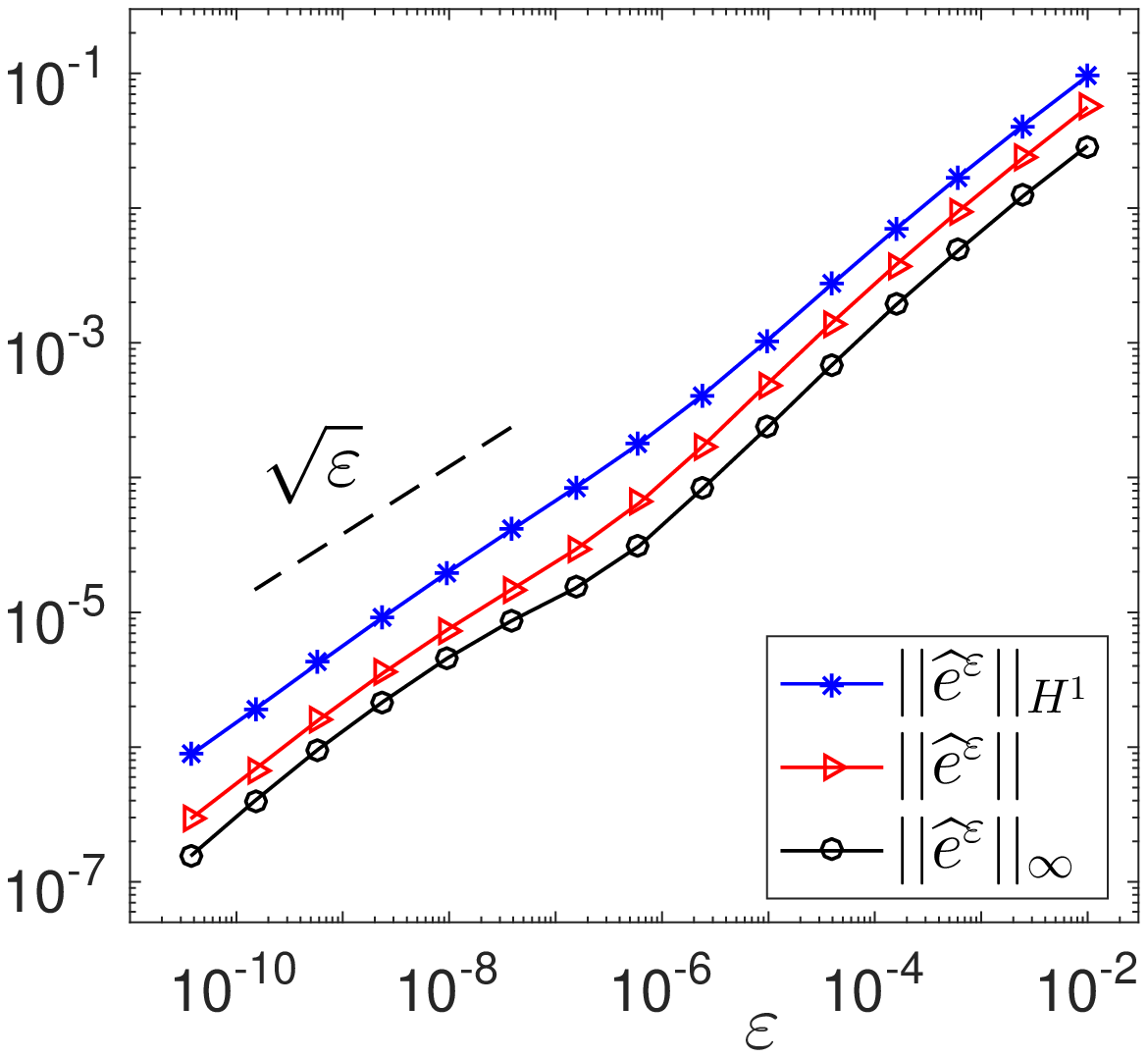,
width=2.5in,height=2.5in}\qquad
\psfig{figure=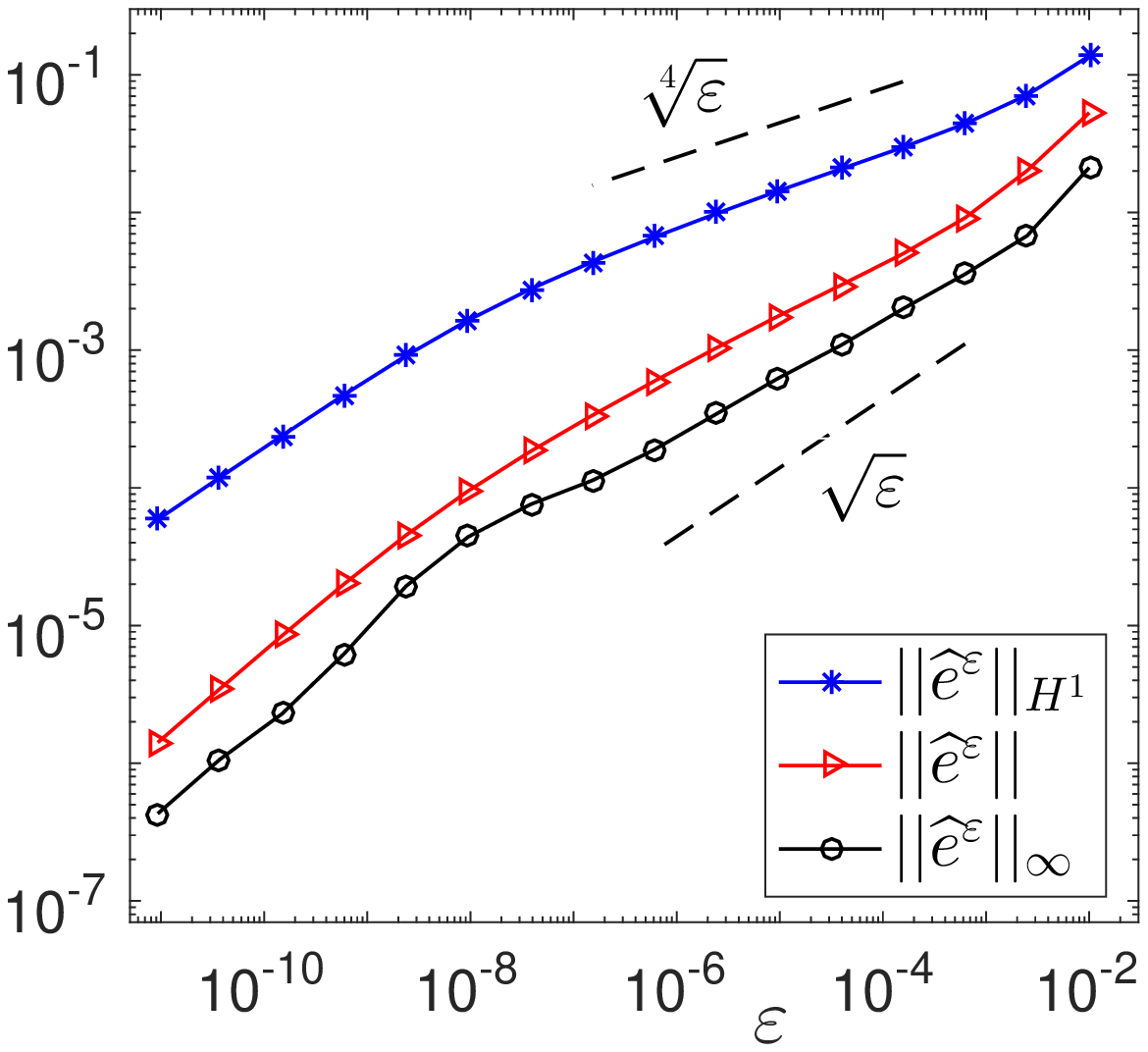,
width=2.5in,height=2.5in}
}
 \caption{Convergence of the RLogSE \eqref{RLogSE_2} to the LogSE \eqref{LSE},  i.e. the error $\widehat{e}^\ep(0.5)$ in different norms vs the regularization parameter $\ep$  for {Case I} (left) and {Case II} (right).}
\label{fig:Conv_RLSEs_Case1_And_2_Mod2}
\end{figure}

From Figs. \ref{fig:Conv_RLSEs_Case1_And_2_Mod1},  \ref{fig:Conv_RLSE_Energy_Error_and_Errors_WRT_t} \&
\ref{fig:Conv_RLSEs_Case1_And_2_Mod2}
and additional numerical results not shown here for brevity,  we can draw the following conclusions: (i) The solution of the RLogSE \eqref{RLSE}
 converges linearly to that of the LogSE \eqref{LSE} in terms
 of the regularization parameter $\ep$ in both $L^2$-norm and $L^\infty$-norm, and respectively, the convergence rate
 becomes $O(\sqrt{\ep})$ in  $H^1$-norm for Case II.
(ii)  The regularized  energy  $E^{\ep}(u^\ep)$ converges linearly to the energy $E(u)$ in terms of $\ep$. (iii) The constant $C$ in
\eqref{Conv_RLSE} may grow linearly with time $T$ and it is independent
of $\ep$.
(iv) The solution of  \eqref{RLogSE_2}
 converges at $O(\sqrt{\ep})$ to that of  \eqref{LSE} in both $L^2$-norm and $L^\infty$-norm, and respectively, the convergence rate
 becomes $O(\ep^{1/4})$ in  $H^1$-norm for Case II. Thus \eqref{RLSE}
 is much more accurate than \eqref{RLogSE_2} for the regularization
 of the LogSE \eqref{LSE}.
 (v) The numerical results agree and confirm our analytical results in Section 2.


\subsection{Convergence rate of the finite difference method}
Here  we test the convergence rate of the SIFD \eqref{scheme}  to the RLogSE \eqref{RLSE} or the LogSE \eqref{LSE} in terms of
mesh size $h$ and time step $\tau$ under any fixed $0<\ep\ll1$ for {Case I}.
Fig. \ref{fig:conv_SIFD_Case1} shows the errors  $\|e^\ep(0.5)\|$ vs time step $\tau$ (with a fixed ratio between mesh size $h$ and time step $\tau$ at $h=75\tau/64$) under different $\ep$. In addition, Table \ref{tab:conv_SIFD_Case1} displays  $\|\widetilde{e}^{\ep}(1)\|$  for varying $\ep$ and $\tau$ \& $h$.

\begin{figure}[htbp!]
\centerline{
\psfig{figure=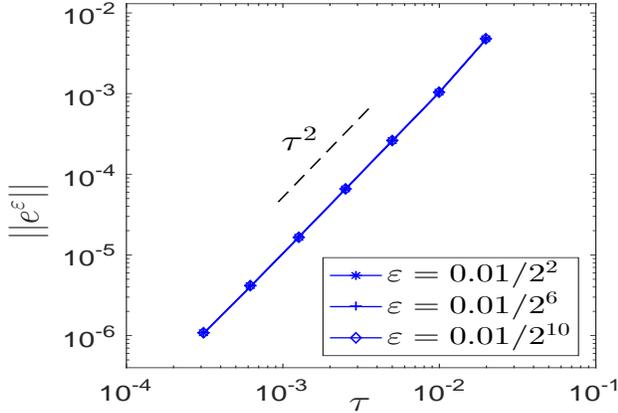,width=3.3in,height=2.2in}
}
 \caption{Convergence of the SIFD \eqref{scheme} to the RLogSE \eqref{RLSE}, i.e. errors $\|e^\ep(0.5)\|$  vs $\tau$ (with $h=75\tau/64$) under different $\ep$  for {Case I} initial data.}
\label{fig:conv_SIFD_Case1}
\end{figure}

\begin{table}[htbp!]
\footnotesize
\tabcolsep 0pt
\caption{Convergence of the SIFD \eqref{scheme} to the LogSE \eqref{LSE}, i.e. $\|\widetilde{e}^{\ep}(1)\|$ for different $\ep$ and $\tau$ \& $h$ for {Case I}.}\label{tab:conv_SIFD_Case1}
\begin{center}\vspace{-0.5em}
\def\temptablewidth{1\textwidth}
{\rule{\temptablewidth}{1pt}}
\begin{tabularx}{\temptablewidth}{@{\extracolsep{\fill}}p{1.38cm}|cccccccccc}
  & $h=0.1$  & $h/2$       & $h/2^2$   & $h/2^3$  & $h/2^4$  & $h/2^5$ &  $h/2^6$ & $h/2^7$  & $h/2^8$ & $h/2^{9}$ \\[0.15em]
 & $\tau=0.1$  & $\tau/2$    & $\tau/2^2$   & $\tau/2^3$  & $\tau/2^4$  & $\tau/2^5$ &  $\tau/2^6$ &  $\tau/2^7$ &  $\tau/2^8$ &  $\tau/2^{9}$  \\[0.3em]
\hline
$\ep$=0.001 & 1.84E-1 & {\bf 4.84E-2} & 1.34E-2 & 5.96E-3 & 4.79E-3 & 4.62E-3 & 4.58E-3 & 4.57E-3 & 4.57E-3 & 4.57E-3  \\ [0.25em]
rate &--&{\bf 1.93}&1.85&1.17&0.31&0.05&	0.01 &0.00&0.00  & 0.00 \\  [0.25em]
\hline
$\ep/4$ & 1.84E-1 & 4.75E-2 & {\bf 1.19E-2} & 3.36E-3 & 1.49E-3 & 1.20E-3 & 1.16E-3 & 1.15E-3  & 1.15E-3 & 1.15E-3	  \\  [0.25em]
rate	&--	&1.96	&{\bf 1.99}	& 1.83	&1.17		&0.31&0.05	 &0.01  &0.00 &0.00	 \\[0.25em]
\hline
$\ep/4^{2}$  		& 1.84E-1& 4.73E-2 & 1.17E-2 & {\bf 2.97E-3} &8.39E-4 & 3.74E-4 & 3.01E-4 & 2.90E-4  &2.88E-4	&2.88E-4  \\  [0.25em]
rate&--	&1.96	&2.01	&{\bf 1.98}	&1.83	&1.17		&0.31 &0.05 &0.01&0.00    \\  [0.25em]
\hline
$\ep/4^3$ 		& 1.84E-1 & 4.72E-2 & 1.16E-2 & 2.91E-3 & {\bf 7.43E-4} & 2.10E-4 & 9.35E-5 & 7.54E-5  	&7.27E-5 &7.21E-5\\  [0.25em]
rate&--&1.96	&2.02	&2.00	&{\bf 1.97}		&1.83&1.16 &0.31  &0.05	&0.01 \\  [0.25em]
\hline
$\ep/4^4$   	 	&1.84E-1 & 4.72E-2 & 1.16E-2 & 2.90E-3 & 7.27E-4 & {\bf 1.86E-4} & 5.24E-5 & 2.34E-5  &1.89E-5	&1.82E-5 \\  [0.25em]
rate	&--	&1.96	&2.02	&2.00	&2.00		&{\bf 1.97} &1.83&1.16 &0.31 &0.05 \\  [0.25em]
\hline
$\ep/4^5$  	& 1.84E-1 & 4.72E-2 & 1.16E-2 & 2.90E-3 & 7.24E-4 & 1.82E-4 & {\bf 4.64E-5} & 1.31E-5  & 5.85E-6	&4.72E-6\\  [0.25em]
rate &--&1.96	&2.02	&2.01&2.00 &1.99 &{\bf 1.97} &1.83&1.16&0.31 \\  [0.25em]
\hline
$\ep_0/4^6$  	& 1.84E-1 & 4.72E-2 & 1.16E-2 & 2.90E-3 & 7.23E-4 & 1.81E-4 & 4.54E-5 & {\bf 1.16E-5} & 3.28E-6	&1.47E-6 \\  [0.25em]
rate&--&1.96	&2.02	&2.01	&2.00		&2.00		&1.99		 &{\bf 1.97}		  &1.83		   &1.16   \\  [0.25em]
\hline
$\ep_0/4^7$  	& 1.84E-1 & 4.72E-2 & 1.16E-2 & 2.89E-3 & 7.23E-4 & 1.81E-4 & 4.52E-5 & 1.14E-5 &{\bf 2.90E-6}	&8.22E-7 \\  [0.25em]
rate	&--		&1.96	&2.02	&2.01	&2.00		&2.00		&2.00		 &2.00		  &{\bf 1.97}		   &1.82  \\  [0.25em]	\end{tabularx}
{\rule{\temptablewidth}{1pt}}
\end{center}
\end{table}

From Fig. \ref{fig:conv_SIFD_Case1},  we can see that
the SIFD \eqref{scheme} converges quadratically at $O(\tau^2+h^2)$
 to the RLogSE \eqref{RLSE} for any fixed $\ep>0$, which confirms
 our error estimates in Theorem \ref{thm1}.
From  Tab. \ref{tab:conv_SIFD_Case1}, we can observe that: (i) the SIFD \eqref{scheme} converges quadratically at $O(\tau^2+h^2)$
to the LogSE \eqref{LSE} only when $\ep$ is sufficiently small, e.g. $\ep\lesssim h^2$ and $\ep\lesssim \tau^2$
(cf. lower triangle below the diagonal in bold letter in Tab. \ref{tab:conv_SIFD_Case1}), and (ii) when $\tau$ \& $h$ is sufficiently small, i.e., $\tau^2\lesssim \ep$ \& $h^2\lesssim \ep$,  the RLogSE  \eqref{RLSE}  converge linearly  at $O(\ep)$ to the LogSE \eqref{LSE} (cf. each column
in the right most of Table \ref{tab:conv_SIFD_Case1}), which confirms
the error bounds in Corollary \ref{cor1}.

\section{Conclusion}
In order to overcome the singularity of the log-nonlinearity
in the logarithmic Schr\"odinger equation (LogSE), we proposed a regularized logarithmic Schr\"odinger equation (RLogSE) with a regularization parameter $0<\ep\ll 1$ and established linear convergence between RLogSE and LogSE
in terms of the small regularization parameter. Then we presented a semi-implicit finite difference method for discretizing RLogSE
and proved second-order convergence rates in terms of mesh size $h$ and
time step $\tau$. Finally, we established error bounds of the semi-implicit finite difference method to LogSE, which depend explicitly on the mesh
size $h$ and time step $\tau$ as well as the small regularization parameter
$\ep$. Our numerical results confirmed our error bounds and demonstrated
that they are sharp.


\bibliographystyle{siam}

\bibliography{biblio}

\end{document}